\documentclass[12pt]{article}
\input epsf.tex

\usepackage{graphicx}
\usepackage{amsmath,amsthm,amsfonts,amscd,amssymb,comment,eucal,latexsym,mathrsfs}
\usepackage{stmaryrd}
\usepackage[all]{xy}

\usepackage{epsfig}

\usepackage[all]{xy}
\xyoption{poly}
\usepackage{fancyhdr}
\usepackage{wrapfig}
\usepackage{epsfig}

\theoremstyle{plain}
\newtheorem{thm}{Theorem}[section]
\newtheorem{prop}[thm]{Proposition}
\newtheorem{lem}[thm]{Lemma}
\newtheorem{cor}[thm]{Corollary}

\theoremstyle{definition}
\newtheorem{rem}{Remark}[section]
\newtheorem{defn}{Definition}
\theoremstyle{remark}

\advance \topmargin by -\headheight
\advance \topmargin by -\headsep
\evensidemargin \oddsidemargin
\def\A{{\mathbb{A}}}  \def\C{{\mathbb{C}}}  \def\E{{\mathbb{E}}} \def\F{{\mathbb{F}}}  \def\H{{\mathbb{H}}}     \def\M{{\mathbb{M}}} \def\N{{\mathbb{N}}}   \def\Q{{\mathbb{Q}}} \def\R{{\mathbb{R}}} \def\SS{{\mathbb{S}}}       \def\Z{{\mathbb{Z}}}

\def\cA{{\mathcal{A}}}   \def\cD{{\mathcal{D}}}   \def\cG{{\mathcal{G}}}     \def\cL{{\mathcal{L}}} \def\cM{{\mathcal{M}}}     \def\cR{{\mathcal{R}}}    \def\cV{{\mathcal{V}}}    

\def\fa{{\mathfrak{a}}}      \def\fg{{\mathfrak{g}}}    \def\fk{{\mathfrak{k}}}  \def\fm{{\mathfrak{m}}}   \def\fp{{\mathfrak{p}}}          

\newcommand\Aut{\operatorname{Aut}}

\newcommand\Hom{\operatorname{Hom}}

\newcommand\Isom{\operatorname{Isom}}

\def\cc{{\curvearrowright}}

\newcommand\norm[1]{\left\|#1\right\|}

\newcommand\abs[1]{\left|#1\right|}

\newcommand\set[1]{\left\{{#1}\right\}}

\begin{document}
\title{Hyperbolic geometry and pointwise ergodic theorems}
\author{Lewis Bowen\footnote{supported in part by NSF grant DMS-1500389}  ~and Amos Nevo\footnote{supported by ISF Moked grant No. 2095/15.}}
\maketitle

\begin{abstract}
We establish pointwise ergodic theorems for a large class of natural averages on simple Lie groups of real-rank-one, going well beyond the radial case considered previously. The proof is based on a new approach 
to pointwise ergodic theorems, which is independent of spectral theory. Instead, the main new ingredient is the use of direct geometric arguments in hyperbolic space. 
\end{abstract}

\noindent

\noindent
\tableofcontents

\section{Introduction}

\subsection{Ergodic subgroups and ergodic theorems}

Let $G$ be a connected simple real Lie group of real rank one with finite center. 
Our purpose in the present paper is to generalize the existing pointwise and maximal ergodic theorems for the ball and shell averages on $G$ well beyond the case of radial averages, using an entirely new approach. 

The ingredients our method utilizes are elementary hyperbolic geometry,  the classical pointwise and maximal ergodic theorems for one-dimensional flows, the Howe-Moore ergodicity theorem, and some variations on the classical ``method of rotation". In particular, our proof is independent 
of any spectral estimates associated with spherical functions on the group $G$. Refined and detailed estimates of spherical functions formed the basis of the only previous proof of pointwise ergodic theorems for radial averages on $G$ \cite{N94}\cite{N97}\cite{NS97}, but reliance on  such estimates necessarily restricts the averages under study to be radial. We remark that our approach in fact extends the range of validity of the radial pointwise ergodic theorems to the space $L\log L$, which is not readily accessible by spectral methods, but the main point in our analysis is the use of geometric ideas to dispense with the assumption of radiality in the pointwise ergodic theorems on $G$.   

To be more precise, let $(X,\mu)$ be a standard probability space and  let $\Aut(X,\mu)$ be the group of all measure-preserving automorphisms of $(X,\mu)$ in which two automorphisms are identified if they agree on a conull set. The group $\Aut(X,\mu)$ is equipped the weak topology under which is it separable and completely metrizable (see \cite{Ke10} for example). A {\bf measure-preserving action of $G$ on $(X,\mu)$} is a continuous homomorphism from $G$ into $\Aut(X,\mu)$. Given such an action, a probability measure $\eta$ on $G$ and a function $f \in L^1(X,\mu)$ on $X$, define $\eta(f) \in L^1(X,\mu)$ by
$$\eta(f)(x) = \int_G  f(g^{-1}x)~d\eta(g).$$
Also let $\E[f|G] \in L^1(X,\mu)$ denote the conditional expectation of $f$ on the sigma-algebra of $G$-invariant Borel sets. This is well-defined up to a measure zero set.

A  1-parameter family of probability measures $\{\eta_r\}_{r > 0 }$ on $G$ is {\bf pointwise ergodic in $L^p$} if for any measure-preserving action $G \cc (X,\mu)$ and any $f \in L^p(X,\mu)$, the averages $\eta_r(f)$ converge pointwise almost everywhere as $r\to\infty$ to conditional expectation $\E[f|G]$.  Such a family is not required to be a semi-group.


The basis of our approach to proving ergodic theorems is the following simple and natural idea. Suppose that $G$ is a locally compact second countable (lcsc) group and $H<G$ is a closed subgroup. We say that $H$ has the {\em automatic ergodicity property} if whenever $G$ acts on a probability space $(X,\mu)$ by measure-preserving transformations ergodically then the action restricted to $H$ is also ergodic. In this case, any pointwise ergodic family of probability measures $\eta_r$ supported on $H$ is a pointwise ergodic family for $G$. It follows that 
for any $g,g^\prime\in G$, the averages $\delta_g\ast \eta_r \ast \delta_{g^\prime}$ satisfy the same conclusion. Given any parametrized family $\delta_{g_b}\ast  \eta_r \ast \delta_{g^\prime_b}$, with $b$ ranging over some lcsc space $B$, 
 the corresponding parametrized pointwise ergodic families can be averaged with respect to a probability measure on $B$. Under suitable natural conditions this gives rise to a host of additional pointwise ergodic families supported on $G$. 
 
 A most significant case where this method can be employed is when $G$ is a simple non-compact algebraic group. Indeed then by the Howe-Moore Theorem any closed noncompact subgroup $H<G$ has the automatic ergodicity property.

Of course, one natural possibility is to choose $H$ as an amenable subgroup of $G$. Then we can use the classical theory of amenable groups to find ergodic sequences in $H$, whose translates $\delta_{g_b}\ast  \eta_r \ast \delta_{g^\prime_b}$, $b\in B$ can then be averaged further on $B$.  For example, when $G$ is a simple non-compact real Lie group, this raises the possibility of proving pointwise ergodic theorems for $G$ by averaging on translates of probability measures on a unipotent subgroup, for example one which is isomorphic to $\R$.  Below we will develop and utilize this approach extensively for the group $SL_2(\R)$ and a  unipotent subgroup $N$.

Furthermore, let us note that parametrized families of translated averages on general, not necessarily amenable subgroups also occur naturally, and we will use the principle stated above in that case too.  For example, we will consider the case of parametrized translates of averages on $SO^0(2,1)$ embedded in $SO^0(n,1)$, which corresponds to embeddings of totally geodesic hyperbolic planes in the $n$-dimensional hyperbolic space. This will allow us to generalize ergodic theorems established for $SO^0(2,1)$ to isometry groups of higher dimensional (real, complex and quaternionic) hyperbolic spaces. 

Thus this approach may be viewed as a generalization of the familiar  ``method of rotation" used extensively in classical analysis and singular integral theory.  


We remark that the approach used in the present paper to prove ergodic theorems for simple real rank one Lie groups was motivated by the method used to prove ergodic theorems for free groups in \cite{BN13}. There the approach is based on considering an appropriately chosen amenable ``measurable subgroup'' of $\F$. This ``subgroup'' is a sub-equivalence relation $\cR$ of the orbit equivalence relation of $\F$ acting on its boundary. Whenever $\F$ acts on a probability space $\F \cc (X,\lambda)$, there is a natural extension $\F \cc (X\times \partial \F, \lambda \times \nu)$ and a sub-equivalence relation $\cR^X$ of the orbit relation on $X\times \partial \F$. It was shown in \cite{BN13} that if the action $\F\cc (X,\lambda)$ is ergodic then the sub-equivalence relation $\cR^X$ has at most 2 ergodic components, which is an analog of the Howe-Moore theorem in this case. Moreover, the subrelation $\cR^X$ is amenable (indeed, it is hyperfinite), and admits ergodic sequences. The radial ergodic theorems for the free groups are then proved by first averaging over finite-sub-equivalence relations of the relation $\cR^X$ and then averaging the result over the boundary.  Note also that this method allows much more general types of averaging sequences to be analyzed similarly, since we can average with respect to a variety of measures on the boundary.

\subsection{Main results}\label{sec:stmt}
As above, let $G$ be a connected simple real Lie group of real rank one with finite center and $G \cc (X,\mu)$ be a measure-preserving action on a standard Borel probability space. Let $KAK=G$ be a Cartan decomposition of $G$ and $A=\{a_t\}_{t\in \R}$ be the Cartan subgroup.  We allow any linear parametrization of the Cartan subgroup $A$, and denote the finite center of $G$ by $Z$.  
We use these coordinates to define the following natural averages on $G$

\begin{defn}\label{def-averages}
Let $U,V \subset K$ be sets of positive measure which are $Z$-invariant, namely $ZU=U$ and $ZV=V$. For $r,\epsilon>0$ let
$$B^{U,V}_r = \{k_1 a_t k_2:~k_1 \in U, t \in [0,r], k_2 \in V\},$$
$$\Sigma^{U,V}_{r,\epsilon} = \{ k_1 a_t k_2:~ k_1 \in U, t \in [r,r+\epsilon], k_2 \in V\}.$$
Let $\sigma^{U,V}_{r,\epsilon}$, $\beta^{U,V}_r$ denote the probability measures on $G$ obtained by restricting Haar measure to $\Sigma^{U,V}_{r,\epsilon}$ and $B^{U,V}_r$ respectively and normalizing to have mass one.
\end{defn}
The following is our main result:
\begin{thm}\label{thm:main}
 Let $G$ be a connected simple real Lie group of real rank one with finite center. For any $\epsilon>0$, 
\begin{enumerate}
\item the families $\{\sigma_{r,\epsilon}^{U,V}\}_{r>0}$ and $\{\beta_{r}^{U,V}\}_{r>0}$ are pointwise and mean ergodic in $L^p$ ($1 < p < \infty$) and in $L\log L$ 
\item the families $\{\sigma_{r,\epsilon}^{U,V}\}_{r>0}$ and $\{\beta_{r}^{U,V}\}_{r>0}$ satisfy the strong $(p,p)$ type maximal inequality $(\forall p>1)$ and the $L\log L$ maximal inequality. 
\end{enumerate}
\end{thm}
The terminology is explained in \S \ref{sec:prelim}.

\begin{rem}
\begin{enumerate}
\item By taking $U=V=K$ we recover the fact that spherical shell averages are pointwise ergodic in $L^p$ for all $p>1$. This was first proven in \cite{N94,N97,NS97} by spectral methods, and the fact that these averages are also pointwise ergodic in $L\log L$ is new.

\item The main novelty occurs when $U$ or $V$ is not equal to $K$. In this case, the averages are referred to as ``bi-sector averages''. Special cases have been proven previously only under the very restrictive hypothesis that $X=G/\Gamma$ is a homogeneous action.

\item We prove a more general result (Theorem \ref{thm:sector2}) in which $U$ and $V$ are replaced with arbitrary bounded $Z$-invariant probability densities on $K$. 

\item Theorem \ref{thm:main} holds for the balls and shells defined by any choice of $G$-invariant Riemannian metric on the symmetric space $G/K$. 

\item The assumption that $U$ and $V$ are $Z$-invariant will be convenient in our arguments, but is not strictly necessary. We will explain this comment further in Remark \ref{general sets}. 

\item A general maximal inequality for not-necessarily-radial averages on connected semisimple Lie groups of any real rank    
was proved in \cite{GN10}. Pointwise ergodic theorems for averages which are $K$-invariant on
one side only were proved under extra spectral assumptions in \cite{Ne98}. 
\end{enumerate}
\end{rem}

{\it Plan of the paper.} In \S 2.1-2.2 we introduce the necessary definitions and notation associated with maximal inequalities and ergodic theorems, and also list some basic standard analytic facts that will be used repeatedly in many of the arguments later on.  \S 3.1-3.2 contain a brief exposition of the classical method of rotations associated with geodesic polar coordinates in Euclidean and hyperbolic space. \S 4.1-4.3  are devoted to proving Theorem \ref{thm:main} in the special case $G=\text{PSL}_2(\R)$.  In \S 5 we prove Theorem \ref{thm:main} by using the $\text{PSL}_2(\R)$ case and the fact that any connected simple adjoint real Lie group $G$ of real rank one contains an embedded subgroup $L$ isomorphic to $\text{PSL}_2(\R)$, conveniently located in $G$. We will then treat the case of finite covers of the adjoint group at the end of \S 5, and in Remark \ref{general sets} comment on  the assumption of $Z$-invariance of the sets $U$ and $V$. 

\subsection{Acknowledgements} The authors would like to thank the referee for several important and useful comments which 
lead to significant improvements in the presentation of the arguments in the paper.  

\section{Preliminaries}\label{sec:prelim}
\subsection{Averaging operators, maximal inequalities and ergodic families}\label{sec:max}
Let $G$ be an lcsc group acting by measure-preserving transformations on a standard Borel probability space $(X,\mu)$. If $\nu$ is any probability measure on $G$ then we also consider $\nu$ to be an operator from $L^1(X,\mu)$ to $L^1(X,\mu)$ via the formula
$$\nu(f)(x)=\int_G f(g^{-1}x)~d\nu(g).$$

{\it Maximal functions and maximal inequalities.} 
Let $r\mapsto \nu_r$, $r>0$ be a 1-parameter family of compactly supported probability measures on $G$. We do not require it to be a semigroup or to consist of absolutely continuous measures on 
$G$.  However, we do require that it is a $w^\ast$-continuous map from $\R_+$ to the space of probability measures ${\cal P}(G)$ on $G$, namely that for any continuous function $F$ on $G$, $r\mapsto \nu_r(F)$ is continuous. We will make this assumption on every 1-parameter family of probability measures without saying so explicitly. The reason this assumption will be useful is as follows.

Let $\M_\nu$ denote the associated maximal operator defined by
$$\M_\nu[f] = \sup_{r\ge 1} \nu_r(|f|).$$
Note that $r\ge 1$ in the definition above instead of $r>0$. This is because we our only interested in limits as $r\to \infty$ and will not be concerned with limits as $r \searrow 0$.

For a general family of averages $\nu_r$, it need not be the case that the maximal function $\M_\nu[f]$ associated with a Borel function $f$ is measurable. However, for an lcsc group $G$ there is a topological model for the action $G\cc X$. In other words, we may assume without loss of generality that $X$ is a compact metrizable space on which $G$ acts jointly continuously and $\mu$ is a standard Borel probability measure on $X$ \cite[Theorem 2.6.6]{BK}. So there exists a subspace ${\cal C}(X)\subset L^\infty(X)$ which is norm dense in every $L^p(X)$, $1\le p < \infty$, such that for every $f\in {\cal C}(X)$ the map $g\mapsto f(g^{-1}x) $ is continuous in $g$ for almost every $x\in X$. Namely ${\cal C}(X)$ is the subspace of continuous functions on $X$. Under the $w^\ast$-continuity assumption for $\nu_r$, for such $f$ the maximal function $\M_\nu[f](x)$ is equal to the supremum of $\nu_q (\abs{f})(x)$ where $q\in \Q\cap \R_+$, so that it is indeed measurable. 


For $f \in L^1(X,\mu)$ and $k\ge 1$, let
$$\|f\|_{L(\log L)^k} = \int_X |f| \left(\log(\max(|f|,1))\right)^k~d\mu$$
and let $L\left(\log L\right)^k(X,\mu) \subset L^1(X,\mu)$ be the set of all functions with $\|f\|_{L\left(\log L\right)^k} < \infty$. This is a vector subspace although $\|\cdot \|_{L\left(\log L\right)^k}$ is not a norm.

We say that a family $\{\nu_r\}_{r>0}$ of Borel probability measures on $G$ satisfies 
\begin{itemize}
\item  the {\bf weak-type $(1,1)$ maximal inequality} if there is a constant $C_1>0$ such that
$$\mu(\{x\in X:~\M_\nu[f] \ge t\}) \le \frac{C_1 \|f\|_1}{t} \quad \forall f\in L^1(X,\mu), t>0,$$
\item  the {\bf strong-type $(p,p)$ maximal inequality} if there is a constant $C_p>0$ such that
$$\|\M_\nu[f] \|_p \le C_p \|f\|_p\quad \forall f\in  L^p(X,\mu).$$
\item  the {\bf strong-type $L\left(\log L\right)^k$ maximal inequality} if there is a constant $C_{L\left(\log L\right)^k}>0$ such that
$$ \|\M_\nu[f] \|_{L^1}  \le C_{L\left(\log L\right)^k} \|f\|_{L\left(\log L\right)^k} \quad \forall f\in L\left(\log L\right)^k(X,\mu)\,.$$
\end{itemize}

{\it Mean and pointwise convergence.} We let $\E[f|G]$ denote the conditional expectation of $f$ on the sigma-algebra of $G$-invariant measurable subsets. We say a family $\{\nu_r\}_{r>0}$ of Borel probability measures on $G$ is
\begin{itemize}
\item {\bf mean ergodic in $L^p$} if $\nu_r(f)$ converges in $L^p$-norm to $\E[f|G]$ as $r \to \infty$ for every $f\in L^p(X,\mu)$;
\item {\bf mean ergodic in $L\left(\log L\right)^k$} if $\nu_r(f)$ converges in $L^1$-norm to $\E[f|G]$ as $r \to \infty$ for every $f\in L\left(\log L\right)^k(X,\mu)$;

\item  {\bf pointwise convergent in $L^p$} (or  $L\left(\log L\right)^k$) if $\nu_r(f)$ converges pointwise a.e. as $r \to \infty$ for every $f\in L^p(X,\mu)$ (or  $L\left(\log L\right)^k$).

\item  {\bf pointwise ergodic in $L^p$} (or  $L\left(\log L\right)^k$) if $\nu_r(f)$ converges pointwise a.e. to $\E[f|G]$ as $r \to \infty$ for every $f\in L^p(X,\mu)$ (or  $L\left(\log L\right)^k$).
\end{itemize}


Finally, we say that $\{\nu_r\}_{r>0}$ is a {\bf good averaging family in $L^p$}, if it satisfies the strong type $(p,p)$-maximal inequality, is mean ergodic for functions in $L^p$ and in addition the family is pointwise ergodic in $L^p$. We define good averaging families in $L\left(\log L\right)^k$ similarly.

When $\nu_r$ is a good averaging family in every $L^p$, $1 < p < \infty$ and also  in $L\log L$ we will abbreviate and say that it is a {\bf good averaging family}. If in addition the family satisfies the weak-type $(1,1)$ maximal inequality then we will say that it is an {\bf $L^1$-good averaging family}. When this holds, it follows that the family is in fact pointwise and mean ergodic in $L^1$. This is one of several useful facts that we will use repeatedly, which we now state.  


\subsection{Standard arguments}

We list the following standard results that will be used frequently below. We start with the following elementary fact. 
\begin{lem}[Domination Lemma]\label{lem:domination}
Suppose $\{\eta_r\}_{r>0}$ and $\{\nu_r\}_{r>0}$ are w*-continuous families of probability measures on $G$ and there is a constant $C>0$ such that $\eta_r \le C \nu_r$ for all $r$. If $\{\nu_r\}_{r>0}$ satisfies either a weak-type $(1,1)$, strong-type $(p,p)$ or $L\left(\log L\right)^k$ maximal inequality then $\{\eta_r\}_{r>0}$ satisfies the same type of maximal inequality.
\end{lem}
We will have occasion to average parametrized families of probability measures on $G$, and thus state the following fact, which is a straightforward consequence of the definitions. 
\begin{lem}[Averaging strong maximal inequalities]\label{max-family}
Let $(Z,\zeta)$ be a standard probability space and let $(z,r) \mapsto \tau_{z,r}$ be a measurable map from $Z \times (0,\infty)$ into the space of Borel probability measures on $G$. Suppose that for each $z\in Z$ the family of averages $\{\tau_{z,r}\}_{r>0}$ is w*-continuous in $r$ and 
satisfies a strong-type $(p,p)$  (or $L\left(\log L\right)^k$) maximal inequality with constants $C_{z,p}$ and moreover the constants $C_{z,p}$ are uniformly bounded for $z\in Z$. Let $\tau_r = \int_{z\in Z} \tau_{z,r}~d\zeta(z)$. Then $\{\tau_r\}_{r>0}$ satisfies the strong-type $(p,p)$ (or $L\left(\log L\right)^k$) maximal inequality.
\end{lem}


We recall that given two bounded Borel measures $\nu$ and $\lambda$ on $G$, their convolution is defined as the functional 
$$\nu\ast \lambda (y)=\int_G \int_G y(gh) d\nu(g)d\lambda(h)\,\,,\,\, \forall y\in C_c(G)$$
where $C_c(G)$ denote the space of compactly supported continuous functions on $G$. Clearly, the support of $\nu\ast \lambda $ is contained in the closure of the product of the supports of $\nu$ and of $\lambda$, and if $\nu$ and $\lambda$ are probability measures, then so is $\nu\ast \lambda$. Convolution is associative and therefore convolution of triples is well-defined by:
$$\nu\ast \alpha\ast\lambda(y)=\int_G \int_G \int_G y(gag^\prime)d\nu(g)d\alpha(a)d\lambda(g^\prime)\,\,,\,\, \forall y\in C_c(G)\,.$$
Below we will often consider maximal inequalities for a family of measures arising as convolutions of probability measures on $G$. We thus state 
\begin{prop}\label{max-conv}
Let $\{\eta_r\}_{r>0}$,$\{\nu_r\}_{r>0}$, $\{\lambda_r\}_{r>0}$ be w* continuous families of compactly supported probability measures on $G$.
\begin{enumerate}
 \item If $\{\eta_r\}_{r>0}$ satisfies the  strong type $(p,p)$ or $L\left(\log L\right)^k$ maximal inequality and $\nu$ and $\lambda$ are fixed (but arbitrary) probability measures, then $\{\nu\ast \eta_r\ast \lambda\}_{r>0}$ satisfies the same maximal inequality.
 \item  If $\{\eta_r\}_{r>0}$, $\{\nu_r\}_{r>0}$ and $\{\lambda_r\}_{r>0}$ each satisfy the strong type $(p,p)$ maximal inequality then so does $\{\nu_r*\eta_r*\lambda_r\}_{r>0}$. 
 \item If $\{\eta_r\}_{r>0}$, $\{\nu_r\}_{r>0}$ and $\{\lambda_r\}_{r>0}$ each satisfy the weak-type $(1,1)$ maximal inequality then $\{\nu_r*\eta_r*\lambda_r\}_{r>0}$ 
 satisfies the $L\left(\log L\right)^3$ maximal inequality. 
\end{enumerate}
\end{prop}


We note that Part (1) follows from Lemma \ref{max-family}, applied to the measure space $(G\times G,\nu\times\lambda)$ and the family defined by $\tau_{(g,g^\prime),r}=\delta_g \ast  \eta_r\ast \delta_{g^\prime}$.  
Part (2) is elementary, since the maximal functions of each of the families $\nu_r$, $\eta_r$ and $\lambda_r$ is itself in $L^p$. Part (3) is proved in \cite{Fav72} (see Theorem 1(ii) and its proof). 
 
 Finally, we recall the following well-known version of the classical Banach principle (see e.g. \cite{Ne05} for complete details).   
\begin{thm}\label{thm:dense}
Suppose there exists a norm-dense subset $\cD \subset L^1(X,\mu)$ of functions such that for every $f\in \cD$, $\eta_r(f)$ converges pointwise a.e. as $r\to\infty$. 
\begin{itemize}
\item If $\{\eta_r\}_{r>0}$ satisfies the weak type $(1,1)$ maximal inequality, then $\{\eta_r\}_{r>0}$ is pointwise and mean convergent in $L^1$.
\item If $\{\eta_r\}_{r>0}$ satisfies the strong type $(p,p)$ (or $L\left(\log L\right)^k$) maximal inequality, then $\{\eta_r\}_{r>0}$ is pointwise and mean convergent in $L^p$ (or $L\left(\log L\right)^k$).
\end{itemize}
If for every $f\in \cD$, $\eta_r(f)$ converges pointwise to the ergodic mean $\E[f|G]$ a.e. as $r\to\infty$, then we can replace ``pointwise and mean convergent'' in the two conclusions above with ``pointwise and mean ergodic''.
\end{thm}



Our discussion below will utilize certain polynomially-weighted versions of Birkhoff's pointwise ergodic theorem. The results we require undoubtedly follows from a suitable weighted ergodic theorem on the real line already in existence, but we have not located a convenient reference, so we include a short self-contained proof of the  simple special case we will use in the following two results. 

\begin{thm}[Polynomially Weighted Birkhoff's Ergodic Theorem]\label{thm:R}
Let $\psi:(0,\infty) \to (0,\infty)$ be a continuous function satisfying 
$$\psi(t) = Ct^\kappa + O\left(t^{\kappa^\prime}\right)$$
for some constants $C>0$ and $\kappa > \kappa^\prime > 0$. Let $\eta$ be the measure on $\R$ defined by $\eta(E) = \int_E \psi(t)~dt$ for Borel $E \subset \R$. Finally, for $T>0$, let $\eta_T$ denote $\eta$ restricted to $[0,T]$ and normalized to have mass 1:
$$\eta_T(E) = \frac{\eta(E\cap [0,T])}{\eta([0,T])}\,\,\,,\,\,\,E \subset \R.$$

Then $\{\eta_T\}_{T>0}$ is an $L^1$-good averaging family for $\R$ (as an additive group). \end{thm}

\begin{proof}

Let $h_t \in \Aut(X,\mu)$ be an $\R$-flow, which we can assume to be ergodic. Let $$\A^\eta_T f=\frac{1}{\eta[0,T]}\int_0^T f\circ h^{-1}_t \,d\eta(t)\quad, \quad \M^\eta[f](x) = \sup_{T>0} \A^\eta_T\abs{f}(x)\,.$$
There exists a dense subset $\cD\subset L^1(X,\mu)$ satisfying the pointwise ergodic theorem, consisting of all functions of the form $f-f\circ h_t+c$ where $f\in L^\infty(X,\mu)$, $t\in \R$ and $c$ is a constant. A standard computation shows that $\A^\eta_T(f-f\circ h_t+c)$ converges pointwise a.e. to $c$ as $T\to \infty$, and a standard argument shows that $\cD$ is dense in $L^2(X,\mu)$ and therefore dense in $L^1(X,\mu)$.

Let $\lambda_T$ denote the uniform probability measure on $[0,T]$. Then there is a constant $C'>0$ such that as measures on $\R$, $\eta_T \le (\kappa + 1)C' \lambda_T$, where $C'$ depends only on the $C$ and the implicit constant in the error term in the formula $\psi(t)=Ct^\kappa + O\left(t^{\kappa^\prime}\right)$. So the Domination Lemma \ref{lem:domination} and the well-known weak type (1,1) maximal inequality for $\{\lambda_T\}_{T>0}$ implies the weak type (1,1) maximal inequality for $\{\A^\eta_T\}_{T>0}$. Similarly, $\M^\eta [f]$ satisfies the strong type $(p,p)$ maximal inequality for all $p>1$. So Theorem \ref{thm:dense} concludes the proof for $\{\A^\eta_T\}_{T>0}$. 

\end{proof}

We now apply the previous theorem to intervals of exponentially increasing size, as follows. 
Assume $  \epsilon  > 0$, $r > 0$, $b > 0$ and  define the measures $\eta_{r,\epsilon}$ on $\R$ by
$$\eta_{r,\epsilon}(E) = \frac{\eta(E \cap [2\sinh br,2\sinh b(r+\epsilon)])}{\eta[2\sinh br,2\sinh b(r+\epsilon)]},~E \subset \R.$$

 \begin{prop}\label{prop:difference}
For any measure-preserving $\R$-action on a probability space, the measures $\{\eta_{r,\epsilon}\}_{r>0}$ constitute an $L^1$-good averaging family, for each fixed $\epsilon > 0$ and $b > 0$. 
\end{prop}
\begin{proof}
Let $h_t \in \Aut(X,\mu)$ be an $\R$-flow, which we can assume to be ergodic. Define operators $\cA_{r,\epsilon}[f]$ for  $f\in  L^p(X,\mu)$ by
$$\cA^\eta_{r,\epsilon}[f](x) =  \eta_{r,\epsilon}(f) = \frac{1}{\eta[2\sinh br,2\sinh b(r+\epsilon)]}  \int_{2\sinh b r}^{2\sinh  b(r+\epsilon)}~f\circ h^{-1}_t~d\eta(t).$$
 Also define
 $$\cM^\eta_{\epsilon}[f](x) = \sup_{r >0} \cA^\eta_{r,\epsilon}[|f|](x).$$

Let us first note the general fact that for any $ \delta > 0$ and for $T > 0$, the averages on $\R$ defined by the normalized restriction of $\eta$ to the set $[0,(1+\delta)T)\setminus [0, T)=[T,(1+\delta) T)$ satisfy the following 
$$\A^\eta_{(1+\delta)T}f(x)-\A^\eta_{T}f(x)=\frac{1}{\eta([0,(1+\delta)T))}\int_0^{(1+\delta)T} f(h^{-1} _tx)~d\eta(t) -\frac{1}{\eta([0,T))}\int_0^{ T} f(h_t^{-1} x) ~d\eta(t)$$
$$ =\frac{\eta([0, T))-\eta([0,(1+\delta)T))}{\eta([0,(1+\delta)T))} \frac{1}{\eta([0,T))}\int_0^{ T} f(h_t^{-1}x ) ~d\eta(t) $$
$$+\frac{\eta([ T,(1+\delta)T))}{\eta([0,(1+\delta)T))}\frac{1}{\eta([T,(1+\delta)T))} \int_{ T}^{(1+\delta)T} f(h_t^{-1} x) ~d\eta(t) 
$$
so that we have the identity 
$$\frac{1}{\eta([T, (1+\delta)T))} \int_{T}^{(1+\delta)T} f(h_t^{-1} x) ~d\eta(t)= \frac{\eta([0,T))}{\eta([T,(1+\delta)T))}
\left(\A^\eta_{(1+\delta)T}f(x)-\A^\eta_{T}f(x)\right) + \A^\eta_{T} f(x)\,.$$
This identity implies immediately that the strong maximal inequalities which are valid for the family $\A^\eta_T$ are valid also for the left hand side, as long as $\frac{\eta([0,T))}{\eta([T,(1+\delta)T))}$
remains bounded. The same argument also establishes the case of the weak-type $(1,1)$-maximal inequality. 

Note that for any given $\delta > 0$, as $T\to \infty$, 
$\A^\eta_{(1+\delta)T}f(x)-\A^\eta_{ T}f(x)$ converges to zero almost everywhere and $\A^\eta_{ T}f(x)$ converges almost everywhere to $\int_X fd\mu$ (in the ergodic case), as follows from Theorem \ref{thm:R}. Therefore the identity shows that pointwise convergence of the left hand side holds as well. 

Let us apply this fact to the choice $T=2\sinh( b r)$ and $\delta$ satisfying $(1+\delta) T=2\sinh (b (r+\epsilon))$. Then 
$$\frac{\eta([0,T))}{\eta([T,(1+\delta)T))}=\frac{\int_0^{2\sinh br} \psi(t) dt}{\int_{2\sinh br}^{2\sinh b(r+\epsilon)}\psi(t)dt}=\frac{(\sinh br)^{\kappa+1}+O\left((\sinh br)^{\kappa^\prime+1}\right)}{(\sinh b(r+\epsilon))^{\kappa+1}-(\sinh br)^{\kappa+1}+O\left((\sinh b(r+\epsilon))^{\kappa^\prime+1}\right)}$$
$$=\frac{1}{\left(\frac{\sinh b(r+\epsilon)}{\sinh br}\right)^{\kappa+1}-1}\left(1+O\left((\sinh br)^{\kappa^\prime -\kappa}\right)\right)
$$

Now using 
$$\frac{\sinh (b (r+\epsilon))}{\sinh (b r)}=
\cosh \frac{b\epsilon}{2}+\coth \frac{br}{2}\sinh \frac{b\epsilon}{2}\ge \cosh \frac{b\epsilon}{2}\ge 1+\frac{b^2\epsilon^2}{8}
\,.$$
we conclude that given fixed $  \epsilon  > 0$ and $b > 0$ we can use the previous identity, 
since the ratio in question remains uniformly bounded. 

We conclude that the family of operators $\cA^\eta_{r,\epsilon}$ satisfies the strong-type $(p,p)$ and $L\log L$ maximal inequalities and  converges pointwise almost everywhere for $f$ in these function spaces. In fact, by the same argument the weak-type $(1,1)$-maximal inequality and pointwise convergence for $L^1$-functions hold as well. 
\end{proof}

Let us now turn to describe in more detail the method of rotations, which will play a significant role in the proof of radial ergodic theorems for simple groups of real rank one which will be established below. 

\section{The classical method of rotations: geodesic polar coordinates }
\subsection{The method of rotations in Euclidean space}\label{sec:Eucl-rotations} Not long after Wiener's proof of the pointwise ergodic  theorem for ball averages on multi-dimensional flows \cite{Wi39}, 
it was pointed out by Pitt \cite{Pi42} that part (but not all) of Wiener's theorem can be established by an argument known as ``the method of rotation" in the context of Calderon-Zygmund theory. We summarize this approach to the pointwise ergodic theorem for Euclidean ball averages, since it includes several arguments and several facts which we will use repeatedly below.  

The main idea is simply to view the normalized uniform measure $\beta^{(n)}_r$ on a ball of radius $r$ in $\R^n$, $n \ge 2$ as a convex average of the normalized (weighted) measures on the intervals $[0,rv]$, with $v\in \SS^{n-1}$ ranging over the unit sphere,  taken with its unique rotation invariant probability measure $m_{\SS^{n-1}}$. Using polar coordinates on $\R^n$, namely representing a general point as $(t,v)$ with $t\ge 0$ and $v\in \SS^{n-1}$, this amounts to writing 
$$\beta^{(n)}_rf(x)=\int_{v\in \SS^{n-1}} \frac{n}{r^n}\int_0^r t^{n-1} f(T_{t,v}^{-1}x)dt dm_{\SS^{n-1}}(v)\,$$
where $T_{t,v}^{-1}x = x-tv$. Now for each fixed $v$, the subgroup $\R\cdot v$ is isomorphic to  $\R$, and the weighted one-dimensional operators  $\cL^v_rf(x)= \frac{n}{r}\int_0^r \left(\frac{t}{r}\right)^{n-1} f(T^{-1}_{t,v}x)dt$ are supported on it. The polynomially weighted Birkhoff's ergodic Theorem \ref{thm:R} implies these operators satisfy the weak-type $(1,1)$ maximal inequality and are pointwise convergent in $L^1$. 

Now, the higher-dimensional ball average $\beta_r^{(n)}$ we are interested in is the average of $\cL^v_rf(x)$ over $v\in \SS^{n-1}$. As a result, norm convergence for the ball averages in $L^p$, $1 \le p < \infty$ follows immediately from norm convergence of $\cL^v_rf(x)$. Similarly, the strong type maximal inequalities in $L^p$, $p > 1$ and $L\log L$ for the ball averages are immediate consequences of the fact that they hold for $\cL^v_rf(x)$ with fixed uniform norm bounds, independent of $v$, using Lemma \ref{max-family}. As to pointwise convergence of the ball averages, it is immediate for bounded functions, for example by applying Lebesgue's Dominated Convergence Theorem to the uniformly bounded family of functions $v\mapsto f(T^{-1}_{t,v}(x))$ in the measure space $L^p(\SS^{n-1}, m_{\SS^{n-1}})$. Pointwise convergence for general functions in $L^p$, $p > 1$ or $L\log L$ then follows using Theorem \ref{thm:dense}. 


Finally, an important additional point is that we must identify the pointwise limit of $\beta_r^{(n)}(f)$, which differs, in general, from the limits of 
$\cL^v_rf(x)$.  However, the limit  function of $\beta^{(n)}_rf(x)$ is in fact invariant under the $\R^n$-action (as noted in Wiener's original argument). Indeed the norm limit of 
$\beta_r^{(n)}(f\circ T_{t,v}) $ is the same for any choice of $v\in \R^n$, as follows easily by comparing the two integrals, and using the asymptotic invariance (=F\o lner) property of Euclidean balls. Thus the limit is the conditional expectation of $f$ with respect to the sigma-algebra of $\R^n$-invariant measurable sets, as stated in the ergodic theorem. 

Note however that the previous argument fails to establish a crucial part of Wiener's ergodic theorem. Namely, it does not establish pointwise almost sure convergence for $L^1$-functions, and cannot be used to prove a weak-type $(1,1)$-maximal inequality. While the family of ball averages in $\R^n$ is the  convex average of the one-dimensional operators $\cL^v_rf(x)$ over $v\in \SS^{n-1}$, and while each one-dimensional family satisfies the weak-type $(1,1)$-maximal inequality, 
this inequality does not average and the inequality for the convex average does not follow.  This limitation will be present throughout our discussion  below. 

\subsection{The method of rotations in non-Euclidean space}

Hyperbolic space $\H^n$ also admits geodesic polar coordinates analogous to those on $\R^n$. 
To describe them more explicitly, recall that the connected component $G$ of the isometry group of $\H^n$ acts transitively and the stability group $K$ of a point $p_0\in \H^n$ 
acts transitively on the unit tangent sphere at $p_0$. Fix a geodesic line $\ell$ in $\H^n$ passing through $p_0$, which is an orbit of a one-parameter group $A=\set{a_r, r\in \R}$ isomorphic to $\R$, so that $\ell=A\cdot p_0$. Every other geodesic through $p_0$ is of the form $kA\cdot p_0$ for some $k\in K$, namely it is the orbit of $p_0$ under the conjugate subgroup $kAk^{-1}$. It follows that the connected component $G$ of the  isometry group of hyperbolic space admits a decomposition of the form 
$G=\text{Iso}^0 (\H^n)=K AK$, and in fact  $G=\{e_G\} \cup KA_+K$, where $A_+=\set{a_r\in A\,;\, r > 0}$. Furthermore, the set $K a_r K $ is mapped under the map $g \mapsto gp_0$ to a sphere of radius $\abs{r}$ with center $p_0$. We let $m_K$ denote the unique Haar probability measure on $K$, and we let $\sigma_r$ be the unique $K$-bi-invariant probability measure on the set $K a_r K$.  The measure $\sigma_r$ coincides of course with the measure $m_K\ast \delta_{a_r}\ast m_K$, where the convolution is defined on $G$ and $\delta_{a_r}$ is the Dirac probability measure at $a_r$.  
\begin{prop}\label{prop:non-Eucl-rotations} 
Let $G=\text{Iso}^0(\H^n)$ act on $(X,\mu)$ by probability-measure-preserving transformations. 
\begin{enumerate}
\item The uniform average $\mu_r=\frac{1}{r}\int_0^r \sigma_t dt$ of the spherical measures $\sigma_t$ is a good averaging family. 
\item Let $\nu$ and $\lambda$ be any two Borel probability measures on the group $K$. Then  
$$r \mapsto \frac{1}{r}\int_0^r (\nu\ast \delta_{a_t}\ast \lambda)dt$$
 is a good averaging family. 
\end{enumerate}
\end{prop} 
\begin{proof}
Part (2) implies part (1) upon taking $\nu=\lambda=m_K$. For part (2), first write using bilinearity of convolution 
$$\frac{1}{r}\int_0^r (\nu\ast \delta_{a_t}\ast \lambda)dt=\nu\ast \left(\frac{1}{r}\int_0^r\delta_{a_t}dt \right)\ast \lambda\,.$$ 
Now the strong maximal inequalities  in $L^p$, $p > 1$ and in $L\log L$ for $A\cong \R$-actions immediately imply the corresponding maximal inequalities for the averages under considerations, by 
 Proposition \ref{max-conv}(1).
 As to pointwise convergence for (say) bounded functions, applying the one-dimensional averages supported on $A\cong \R$ to $\lambda f$ we can conclude pointwise convergence as $r\to \infty$. Now the averaging operator $\nu$ maps a pointwise convergent family of bounded functions to another pointwise convergent family of bounded functions, for example using Lebesgue's Dominated Convergence Theorem as in the Euclidean case explained above. Using Theorem \ref{thm:dense} again, pointwise almost sure and norm convergence for the desired averages follow. Finally, the identification of the limit requires an additional argument, since asymptotic invariance arguments are absent in our non-amenable group. In the present case, it is well known that $A\subset \text{Iso}^0(\H^n)$ acts ergodically in every ergodic $\text{Iso}^0(\H^n)$-space by the Howe-Moore theorem \cite{HM79}. Thus $\frac{1}{r}\int_0^r  (\lambda f)\circ a_t^{-1} dt$ converges pointwise a.e. to the ergodic mean $\E[f|G]$ by the one-dimensional pointwise ergodic theorem, and applying $\nu$ to this pointwise convergent family yields the desired conclusion.   

\end{proof}
Let us note that the measure $\mu_r$ when projected to $\H^n$ is supported in the ball of radius $r$ with center $p_0$. It gives a measure which is absolutely continuous with respect to the Riemannian measure $\beta_r$ on the ball, namely the measure which is the normalized restriction of the isometry-invariant measure on hyperbolic space to the ball. However, $\mu_r$ and $\beta_r$ are radically different measures, since the radial density of $\beta_r$ is given by  
\begin{equation}\label{eq:beta}
\beta_r=\frac{\int_0^r \sigma_t (\sinh t)^{n-1}dt}{\int_0^r (\sinh t)^{n-1}dt}\,.
\end{equation} 
 Thus the application of the classical method of rotation using geodesic polar coordinates leaves much to be desired, in the case of $\text{Iso}^0 (\H^n)$. We are interested in establishing ergodic theorems for the ball measures $\beta_r$, which arise intrinsically from hyperbolic geometry, and not just for the uniform average of the sphere measures $\sigma_t$. 
 
 Much of the present paper is based on the following two observations. First, this goal can be still be achieved by  the method of rotation applied to averages on horospheres, rather than geodesics. Second, averages on horospheres can be used to establish convergence even for natural non-radial averages as well, using more refined geometric arguments. 
 We now turn to demonstrate these observations and their consequences in the case of the hyperbolic plane, which is fundamental to the developments that follow.

\section{Ergodic theorems for the isometry group of the hyperbolic plane}

Let $\H^2$ denote the hyperbolic plane, equipped with a Riemannian metric of constant negative sectional curvature. We identify $\text{PSL}_2(\R)$ with the group of orientation preserving isometries of $\H^2$ in the usual way. For $r > 0$, $\epsilon>0$, define the annuli 
$$\Sigma_{r,\epsilon}=\{g\in \text{PSL}_2(\R):~ d(gp_0,p_0) \in [r,r+\epsilon]\}$$
where $d(\cdot,\cdot)$ is the invariant Riemanian (=hyperbolic) distance in $\H^2$.  Let $\sigma_{r,\epsilon}$ be the probability measure on $\Sigma_{r,\epsilon}$ obtained by normalizing the restriction of Haar measure. Also let $\beta_r$ be the probability measure on $\{g\in \text{PSL}_2(\R):~ d(gp_0,p_0)  \le r\}$ (for some $p_0 \in \H^2$) obtained by normalizing the restriction of Haar measure. 


We will start by proving the following radial ergodic theorem, which will be followed later on by a non-radial generalization. 
%

\begin{thm}\label{thm:sl2} For $G=\text{PSL}_2(\R)$, 
the families $\{\beta_r\}_{r>0}$ and $\{\sigma_{r,\epsilon}\}_{r>0}$ are good averaging families (for any fixed $\epsilon > 0$). 
\end{thm}

\subsection{The upper half plane model}\label{sec:model}

For convenience we let $\H^2$ denote the upper half plane $\H^2=\{ x+iy \in \C:~y>0\}$ with the Riemannian metric $\cG$ given by
$$\cG(w,z) = \langle w,z\rangle/y$$
for any vectors $w,z$ in the tangent space of $x+iy \in \H^2$, where the inner product on the right hand side is the usual inner product in Euclidean space. With this metric, $\H^2$ is a model of the hyperbolic plane (so this is consistent with previous notation). It is a complete simply connected Riemannian manifold with constant sectional curvature -1. 
We denote the associated Riemannian distance in $\H^2$ by $d(\cdot,\cdot)$ and note the following well known formula :
\begin{equation}\label{distance}
\cosh(d( x_1 + y_1i, x_2+y_2i)) = 1 + \frac{ |x_1-x_2|^2 + |y_1-y_2|^2}{2y_1y_2}\,.
\end{equation}

The group $\text{SL}_2(\R)$ acts on $\H^2$ by fractional linear transformations:
\begin{displaymath}
\left(\begin{array}{cc}
a & b \\
c & d \end{array}\right) z = \frac{az +b }{cz+d}\,,
\end{displaymath}
and this action preserves the Riemannian metric and the distance. Because the center $\{\pm I\} \le \text{SL}_2(\R)$ acts trivially, this induces an action of $\text{PSL}_2(\R) =\text{SL}_2(\R)/\{\pm I\}$. It is well-known  that this gives an isomorphism of $\text{PSL}_2(\R)$ with $\Isom^+(\H^2)$, the group of orientation-preserving isometries of the hyperbolic plane.

For $r, t \in \R$, $\theta \in [0,2\pi)$ let 
\begin{displaymath}
k_\theta = \left(\begin{array}{cc}
\cos(\theta) & -\sin(\theta) \\
\sin(\theta) & \cos(\theta) \end{array}\right), ~a_r = \left(\begin{array}{cc}
e^{r/2} & 0 \\
0 & e^{-r/2} \end{array}\right), ~n_t = \left(\begin{array}{cc}
1& t \\
0 & 1 \end{array}\right)
\end{displaymath}
and
$$SO_2(\R)=\{k_\theta\}_{\theta \in \R}, A=\{a_r\}_{r\in \R}, N=\{n_t\}_{t\in \R}.$$
We note that the double cover $SL_2(\R)\to \text{PSL}_2(\R)$ is injective when restricted to $A$ and $N$, and is a double cover when restricted to $SO_2(\R)$. We denote by $K$ the image of $SO_2(\R)$. 

Because the isometry group acts transitively we may assume without loss of generality that $p_0=i$. 
Note that $K$ is the stabilizer of $p_0$ and $d(a_r p_0,p_0)=|r|$ according to the distance formula (\ref{distance}).  The next lemma is central to our approach.


\begin{lem}\label{distance on horocycle}\label{lem:convert1}
Let $r,t>0$. The following are equivalent.
\begin{enumerate}
\item $d(n_t p_0, p_0) = r$;
\item $Ka_rK = Kn_tK$;
\item $\cosh(r)=1+\frac{t^2}{2}$;
\item $t=2\sinh(r/2)$.
 \end{enumerate}
\end{lem}
\begin{proof}
 Since $K$ is the stabilizer of $p_0$, $d(a_rp_0,p_0)=r$ and $\text{PSL}_2(\R)$ acts simply transitively on the unit tangent bundle of $\H^2$, it follows that $\{Ka_rK\}_{r\ge 0}$ is a parametrization the space of double cosets of $K$ in $G$. Using the distance formula (\ref{distance}) :
$$\cosh(d( x_1 + y_1i, x_2+y_2i)) = 1 + \frac{ |x_1-x_2|^2 + |y_1-y_2|^2}{2y_1y_2}\,$$
the equivalence of (1) and (3) follows  upon substituting $x_1=0, x_2=t$, and $y_1=y_2=1$. The equivalence of (3) and (4) is elementary.
\end{proof}

\begin{rem}\label{uniqueness-PSL(2)}
We note that in the group $G=PSL_2(\R)$, the Cartan decomposition $g=ka_r k'= ua_s u'$ with $k,k',u, u'\in K\subset PSL_2(\R)$, and both $r$  and $s$ are positive, is {\it unique}. 
Indeed, the absolute values of $r$ and $s$  are equal to $d(gi,i)$ (where $K$ stabilizes $i$), hence  $r=s$ if both are positive. Note that otherwise uniqueness fails, because $w_0  a_t w_0^{-1}=a_{-t}$ if $w_0$ is the Weyl group element. Now looking at $gi$, since $K\cdot i = i$ we have $ka_r i=ua_r i$. So $u^{-1}k$ stabilizes the point $a_r i \neq i$ and also the point $i$.  Since it acts as a rotation with center $i$, it must be the identity. Then also $k'=u'$ and the representation is unique.
\end{rem}

\begin{rem}\label{curvature}
Let us briefly digress and note the following facts, which will be used in our discussion later on. Multiplying the Riemannian metric $\cG$ used above by a positive scalar $c$ gives rise to a different Riemannian manifold, namely the unique complete simply connected Riemannian manifold of constant sectional curvature $-1/\sqrt{c}$.  The distance between two points in the upper half plane in the associated metric $d_c$ is $d_c(p,q)=cd(p,q)$. The geodesic passing through $p_0$ remain the same, and $d_c(a_r p_0,p_0)= c \abs{r}=c d(a_r p_0,p_0)$.
Thus the Riemannian metric $c \cG$ gives rise to geodesics through $p_0$ which are parametrized by $d_c(a_{r/c}p_0,p_0)=r$. Thus changing the curvature on the hyperbolic plane amounts to reparametrizing the geodesics, and hence also the radii of balls and shells.
It follows that prove Theorem \ref{thm:sl2} it suffices to establish the result when the curvature is $-1$.  

Furthermore, it follows that a version of Lemma \ref{lem:convert1} still holds for the metric $d_c$, in the following modified form: $ d_c(n_tp_0, p_0)=r  \iff Ka_{r/c}K=Kn_t K\iff \cosh \frac{r}{c}=1+\frac{t^2}{2}\iff t=2\sinh \frac{r}{2c}\,.$
\end{rem}

Let $\eta$ denote the measure on $N$ given by
$$\eta(E) = \int 1_E(n_t) |t|~dt.$$

\begin{lem}[$KNK$ decomposition]\label{lem:KNK-sl2r}
Let $m_K$ denote Haar probability measure on $K$ and $m_G$ denote Haar measure on $G=PSL_2(\R)$ normalized so that 
$$m_G(\{g\in G:~d(gp_0,p_0) \le r\}) = 2\pi(\cosh r -1)$$
is the same as the area of the ball of radius $r$ in $\H^2$. Then
$$\frac{1}{\pi}m_G =  m_K * \eta * m_K.$$
\end{lem}

\begin{proof}
Since both $m_G$ and $m_K*\eta*m_K$ are bi-$K$-invariant, it suffices to prove 
$$\frac{1}{\pi}m_G(E) =  m_K * \eta * m_K(E)$$
for bi-$K$-invariant subsets $E \subset G$ (in other words, sets satisfying $E=KEK$). Because balls centered at $p_0$ generate such sets, it suffices to prove  
$$\frac{1}{\pi}m_G(B_r(p_0)) =  m_K * \eta * m_K(B_r(p_0))$$
where $B_r(p_0) = \{g\in G:~d(gp_0,p_0) \le r\}$. From the previous lemma, and the fact that $d(n_t p_0,p_0)=d(n_{-t}p_0,p_0)$ :
\begin{eqnarray*}
m_K * \eta * m_K(B_r(p_0)) &=& \eta(\{n_t:~ |t| \le 2\sinh(r/2) \}) \\
&=& 2\int_0^{2\sinh(r/2)} t ~dt = (2\sinh(r/2))^2 = 2\cosh(r)-2 \\
&=& \frac{1}{\pi} m_G(B_r(p_0)).
\end{eqnarray*}
\end{proof}

\begin{defn}\label{eta-r-e} For $r,\epsilon>0$, let $\eta_{r,\epsilon}$ be the probability measure on $N$ given by
$$\eta_{r,\epsilon}(E) = \frac{ \eta( E \cap \{n_t:~t \in [2\sinh(r/2), 2\sinh((r+\epsilon)/2)]\})}{ \eta( \{n_t:~t \in [2\sinh(r/2), 2\sinh((r+\epsilon)/2)]\})}.$$
\end{defn}
 Lemma \ref{lem:convert1} together with Lemma \ref{lem:KNK-sl2r} imply
\begin{eqnarray}\label{eqn:convolve1}
\sigma_{r,\epsilon} = m_K * \eta_{r,\epsilon} * m_K.
\end{eqnarray}

\begin{proof}[Proof of Theorem \ref{thm:sl2}]
By the polynomially weighted Birkhoff ergodic Theorem \ref{thm:R}, $\{\eta_{r,\epsilon} \}_{r>0}$ is an $L^1$-good averaging sequence for $N$. By the Howe-Moore Theorem, $N$ has the automatic ergodicity property as a subgroup of $G$. Therefore $\{\eta_{r,\epsilon} \}_{r>0}$ is an $L^1$-good averaging  family for $G$. Proposition \ref{max-conv} implies $\{\sigma_{r,\epsilon}\}_{r>0}$ satisfies the strong $(p,p)$-type maximal inequality and the $L\log L$ maximal inequality. Since $\{\eta_{r,\epsilon} \}_{r>0}$ is pointwise ergodic for bounded functions, the Bounded Convergence Theorem implies $\{m_K*\eta_{r,\epsilon}*m_K \}_{r>0}$ is also pointwise ergodic for bounded functions. Since $L^\infty(X,\mu)$ is dense in $L^1(X,\mu)$ (for any probability space $(X,\mu)$), Theorem \ref{thm:dense} now implies $\{m_K*\eta_{r,\epsilon}*m_K \}_{r>0}$ is a good averaging family. Equation (\ref{eqn:convolve1}) finishes the proof. The proof that $\{\beta_r\}_{r>0}$ is a good averaging family is similar. By Remark \ref{curvature}, it suffices to consider the case of curvature $-1$. 
\end{proof}

We now formulate the following generalization of Theorem \ref{thm:sl2} pertaining to non-radial averages. This result will be crucial to our discussion below of sector averages.



\begin{prop}\label{SL2 : general averages}
\begin{enumerate}
\item Let $\nu$ and $\lambda$ be  arbitrary Borel probability measures on $K$. Then 
$\nu\ast \eta_{r,\epsilon}\ast \lambda$ 
is a good averaging family. 
\item Assume further that $\nu$, $\lambda$, $\nu_r$ and $\lambda_r$ are probability measures on $K$, each is absolutely continuous to the Haar measure and $\frac{d\nu_r}{dm_K} \to \frac{d\nu}{dm_K}$, and $\frac{d\lambda_r}{dm_K}\to \frac{d\lambda}{dm_K}$ in the $L^1(K,m_K)$-norm as $r\to\infty$.  If the family 
$\{\nu_r \ast \eta_{r,\epsilon} \ast \lambda_r\}_{r>0}$ satisfies the strong maximal inequalities in $L^p$, $ 1<p< \infty$ and in $L\left(\log L\right)$, then it is a good averaging family. 
\end{enumerate}
\end{prop}
\begin{proof}
By the polynomially weighted Birkhoff ergodic Theorem \ref{thm:R}, $\{\eta_{r,\epsilon} \}_{r>0}$ is an $L^1$-good averaging sequence for $N$. The Howe-Moore Theorem implies $\{\eta_{r,\epsilon} \}_{r>0}$ 
is a good averaging family as a family of measures on $\text{PSL}_2(\R)$.  Proposition 
 \ref{max-conv} implies $\nu\ast \eta_{r,\epsilon}\ast \lambda$ 
is a good averaging family. This proves (1).


As to part (2), given the maximal inequalities assumed in it, by Theorem \ref{thm:dense}  it suffices to prove pointwise convergence for the dense subspace $ L^\infty(X)$.   For every bounded function $f$ and for almost every $x\in X$, 
$\abs{\lambda_r f(x)-\lambda f(x)}\le \norm{\lambda_r-\lambda}_{L^1(K)}\norm{f}_{L^\infty(X)}$, so that $\norm{\lambda_r f-\lambda f}_{L^\infty(X)}\le \norm{\lambda_r-\lambda}_{L^1(K)}\norm{f}_{L^\infty(X)}$. A similar statement holds for $\nu_r-\nu$. 

Therefore, for almost every $x\in X$ 
\begin{eqnarray*}
&&\abs{\left(\nu_r \ast \eta_{r,\epsilon} \ast \lambda_r\right)f(x)- \left(\nu \ast \eta_{r,\epsilon} \ast \lambda\right) f(x)} \\
&\le& \abs{\left((\nu_r -\nu)\ast \eta_{r,\epsilon} \ast \lambda_r \right)f(x)} + \abs{\left(\nu \ast \eta_{r,\epsilon} \ast( \lambda_r- \lambda)\right)f(x)}\\
&\le&\norm{\frac{d\nu_r}{dm_K}-\frac{d\nu}{dm_K}}_{L^1(K)}\norm{ \left(\eta_{r,\epsilon} \ast \lambda_r \right)f}_{L^\infty(X)} + \norm{\nu \ast \eta_{r,\epsilon}}_{L^\infty(X)\to L^\infty(X)}\norm{\lambda_r f-\lambda f}_{L^\infty(X)}\\
&\le& \norm{\frac{d\nu_r}{dm_K}-\frac{d\nu}{dm_K}}_{L^1(K)}\norm{f}_{L^\infty(X)}+ 
\norm{\frac{d\lambda_r}{dm_K}-\frac{d\lambda}{dm_K}}_{L^1(K)}\norm{f}_{L^\infty(X)}
\end{eqnarray*} 
and the limit of the latter expression as $r\to \infty$ is zero by assumption. By part (1), $\left(\nu \ast \eta_{r,\epsilon} \ast \lambda\right) f$ is pointwise convergent a.e. to the ergodic mean. So the computation above implies $\left(\nu_r \ast \eta_{r,\epsilon} \ast \lambda_r\right) f$ is also pointwise convergent a.e. to the ergodic mean. 
\end{proof}

\subsection{From horocycle averages to bi-sector averages}\label{sec:AtoN}
In the present subsetion we consider $G=PSL_2(\R)$, whose maxial compact subgroup is denoted by $K$. 
Recall that $A=\{a_t\}_{t\in \R} \le \text{PSL}_2(\R)$ is a 1-parameter subgroup satisfying $d(a_tp_0,p_0)=\abs{t}$. For $r,\epsilon>0$, let $\alpha_{r,\epsilon}$ denote the probability measure on $A\subset G$ given by
$$ \alpha_{r,\epsilon} = \frac{\int_r^{r+\epsilon} \sinh(t)\delta_{a_t}~dt}{\int_r^{r+\epsilon} \sinh(t)~dt}.$$
For example, note that $m_K\ast \alpha_{r,\epsilon} \ast m_K = \sigma_{r,\epsilon}$ where $m_K$ denotes Haar probability measure on $K$.

\begin{thm}\label{thm:sector1}
 If $\nu,\lambda<<m_K$ are probability measures with densities $\frac{d\nu}{dm_K},\frac{d\lambda}{dm_K} \in L^\infty(K,m_K)$ and $\epsilon>0$ then $\{\nu\ast \alpha_{r,\epsilon}\ast \lambda\}_{r>0}$
is a good averaging family. 
\end{thm}

A special case of this theorem pertains to {\bf ``bi-sector averages''}, defined as follows. For Borel subsets $U,V \subset K$ with positive Haar measure and $r,\epsilon>0$, let $G^{U,V}_{r,\epsilon}$ be the set of all $g\in G$ such that $g=ua_t v$ for some $u\in U, t \in [r,r+\epsilon]$ and $v \in V$. Let  $\sigma^{U,V}_{r,\epsilon}$ be the measure on $G$ equal to Haar measure restricted to $G^{U,V}_{r,\epsilon}$ and normalized to have total mass one. 

     \begin{cor}\label{cor:sector}
     For any Borel subsets $U,V \subset K$ with positive Haar measure and any $\epsilon>0$, $\{\sigma^{U,V}_{r,\epsilon}\}_{r>0}$ is a good averaging family.
      \end{cor}
  
  \begin{proof}[Proof of Corollary \ref{cor:sector} from Theorem \ref{thm:sector1}]
  This follows immediately from Theorem \ref{thm:sector1} by setting $\nu=m_K(U)^{-1}\chi_U$ and $\lambda=m_K(V)^{-1}\chi_V$. Indeed then $\nu\ast \alpha_{r,\epsilon} \ast \lambda= \sigma^{U,V}_{r,\epsilon}$.
  \end{proof}
  
We will derive Theorem \ref{thm:sector1} from a special case of Corollary \ref{cor:sector} which we prove after the next lemma.
  
    \begin{lem}\label{lem:strong-domination}
  Suppose $\{\tau_r\}_{r>0}$ and $\{\tau^\prime_r\}_{r>0}$ are families of probability measures on $G$ and $C_r>1$ satisfies:
  \begin{itemize}
  \item $\tau_r \le C_r \tau^\prime_r$,
  \item $C_r$ is uniformly bounded,  and $C_r \to 1$ as $r \to \infty$,
  \item $\{\tau^\prime_r\}_{r>0}$ is a good averaging family.
  \end{itemize}
  Then $\{\tau_r\}_{r>0}$ is also a good averaging family.
  \end{lem}
 
 \begin{proof}
 It follows from the Domination Lemma \ref{lem:domination} that $\{\tau_r\}_{r>0}$ satisfies the strong-type $(p,p)$ maximal inequalities for $1<p\le \infty$ and the strong-type $L\log L$ maximal inequality. Let $f\in L^\infty(X,\mu)$ be nonnegative. Then $\tau^\prime_r(f)$ converges to $\E[f|G]$ pointwise a.e. as $r \to \infty$. Since $\tau_r(f) \le C_r \tau^\prime_r(f)$ and $C_r \to 1$ as $r\to\infty$ it follows that $\limsup_{r\to\infty} \tau_r(f)(x) \le \E[f|G](x)$ for a.e. $x$. Since $\tau_r$ preserves the $L^1$-norm of non-negative functions, $ \int \tau_r(f)(x) ~d\mu(x) = \int \E[f|G](x)~d\mu(x)$. By Fatou's Lemma,
 $$\int \E[f|G](x)~d\mu(x)=\limsup_{r\to\infty} \int \tau_r(f)(x)~d\mu(x)$$
 $$ \le \int \limsup_{r\to\infty} \tau_r(f)(x)~d\mu(x)\le \int \E[f|G](x)~d\mu(x).$$
 Thus $\limsup_{r\to\infty} \tau_r(f) = \E[f|G]$ a.e. Since $\|f\|_\infty - f$ is also non-negative and bounded, the same argument gives $\limsup_{r\to\infty} \tau_r(\|f\|_\infty-f) = \E[(\|f\|_\infty - f)|G]$ a.e. Since $\|f\|_\infty$ is constant, this implies $\liminf_{r\to\infty} \tau_r(f) = \E[f|G]$ a.e. and therefore $\tau_r(f)$ converges pointwise a.e. to $\E[f|G]$ as $r \to \infty$. By decomposing an arbitrary $f \in L^\infty(X,\mu)$ into real and imaginary parts and then into positive and negative parts, we see that $\tau_r(f)$ converges pointwise a.e. to $\E[f|G]$ as $r \to \infty$. Since $L^\infty$ is dense in $L\log L$ and in $L^p$ ($1<p<\infty$) the lemma now follows from Theorem \ref{thm:dense}.
  \end{proof} 
  
  \begin{thm}\label{lem:open}
  If $U,V \subset K$ are compact sets with positive Haar measure and $\epsilon>0$, then the family $\{\sigma^{U,V}_{r,\epsilon}\}_{r>0}$ is a good averaging family. 
      \end{thm}

 The proof of Theorem \ref{lem:open} is based on the following geometric Lemma.    As noted in  Lemma \ref{lem:convert1}, 
  for every $r>0$ there is a unique positive $t=t(r)$ with $Kn_tK = Ka_rK$, namely $t=2\sinh (r/2)$. Using Remark \ref{uniqueness-PSL(2)}, 
let $w_r,w^\prime_r \in K$ be the unique elements with $n_t = w_r a_r w^\prime_r$. We will utilize the following observation on the angular components in this decomposition. 
  

  \begin{lem}\label{NtoA}
  For $t, r > 0$, the polar coordinates decomposition $n_t = w_r a_r w^\prime_r$ in $\text{Isom}^+(\H^2)=\text{PSL}_2(\R)$ satisfies that the (unique) component 
   $w_r$ converges to the identity element as $r\to\infty$ while the (unique) component $w^\prime_r$ converges to the $180^o$ rotation in $\H^2$ with center $p_0$. Furthermore, $w_r$ and $w_r^\prime$ are continuous functions of $t$.

  \end{lem}
   \begin{proof}
   
    We use notation as in \S \ref{sec:model}; in particular $\H^2$ denotes the upper half plane model and we identify $\text{Isom}^+(\H^2)$ with $\text{PSL}_2(\R)$ through the latter's action on $\H^2$ by fractional linear transformations. Then the action of $w_r$ is given by the matrix $k_\theta$ for some $\theta=\theta_r$ and the action of $w^\prime_r$ by the matrix $ k_{\theta^\prime}$ for some $\theta^\prime=\theta^\prime_r$, 
with $ k_{\theta},k_{\theta^\prime} \in SO(2,R)\subset SL_2(\R)$. Therefore:
    
    $$n_t i=t+i=w_ra_r i=\frac{e^{r/2} i\cos \theta-e^{-r/2}\sin \theta}{e^{r/2} i \sin \theta+ e^{-r/2}\cos \theta}=$$
   $$=\frac{(e^r-e^{-r} )\sin\theta \cos \theta}{e^r\sin^2\theta+e^{-r}\cos^2\theta}+\frac{i}{e^r\sin^2\theta+e^{-r}\cos^2\theta}.
   $$
   Thus $e^r\sin^2\theta+e^{-r}\cos^2\theta=1$ and since $r\to \infty$, we have $\sin^2\theta\to 0$ and so $\cos^2 \theta\to 1$. 
Thus $w_r=\set{\pm  k_{\theta}}$ converges to $\set{\pm I}$ in $SL_2(\R)/\set{\pm I}=PSL_2(\R)$ as $r\to \infty$. 
 
 For future reference we note that since $d(n_t i,i)=d(n_{-t}i,i)$ it is geometrically clear that
   $t=(e^r-e^{-r})\sin\theta_r \cos \theta_r$ can be solved uniquely for any $t\in \R\setminus\set{0}$, and for $\pm t$ the same value of $e^{r/2}$ is obtained, together with the values $ \theta_r$ and $-\theta_r$.

   Writing $n_{-t}=n_t^{-1}=(w_r^\prime)^{-1}a_r^{-1} w_r^{-1}$, we have $n_{-t}i=-t+i=(w_r^\prime)^{-1}a_r^{-1}i$.   Substitution in the foregoing explicit formula shows that  $e^{-r}\sin^2\theta^\prime_r+e^{r}\cos^2\theta^\prime_r=1$, and thus $\cos^2 \theta^\prime_r\to 0$ and $\sin^2\theta^\prime_r\to 1$, as $r\to \infty$. 
 We conclude that 
 $ w_r^\prime \to \set{\pm \left(\begin{smallmatrix}0& -1\\ 1& 0\end{smallmatrix}\right)}$
 in $SL_2(\R)/\set{\pm 1}$ as $r\to \infty$. Note that since $SO_2(\R)\to K=SO_2(\R)/\set{\pm I}$ is a double cover, the matrix $\left(\begin{smallmatrix}0& -1\\ 1& 0\end{smallmatrix}\right)$, which defines a $90^o$ rotation in the Euclidean plane, is mapped to a $180^o$ rotation of the non-Euclidean plane. 
 
 Finally, the continuity of $w_r$ and $w_r^\prime$ (and of course, $a_r$ as well) is evident from the foregoing explicit formulas given for $\cos \theta$ and $\cos \theta^\prime$ above. 
 
   \end{proof}
   
%
     
   \begin{rem}\label{integral}
 It is elementary to check that the map $ N_+\cong(0,\infty)\to (0,\infty)\cong A_+$ given by $t\mapsto 2 \sinh^{-1} (t/2)=s(t)=s$ maps the measure $\eta(t)=tdt$ to the measure $\sinh (s) ds$.
    Consider the map $J : K\times N_+\times K\to K \times A_+ \times K$ given by $J(k, n_t, k^\prime)=(kw_s, a_s, w_s^\prime k^\prime)$, where $n_t=w_s a_s w_s^\prime$ is the unique Cartan coordinates  representation of $n_t$ with $s=s(t)=2 \sinh^{-1} (t/2) > 0$ (using Lemma \ref{lem:convert1} and Remark  \ref{uniqueness-PSL(2)}). 
    Given any compact sets of positive measure $U, V\subset K$, the measure $\nu=\chi_U dm_K \times tdt\times \chi_V dm_K$ 
    satisfies 
  $$J_\ast(\chi_U dm_K \times tdt\times \chi_V dm_K)=\int_0^\infty \chi_{Uw_s}(k)dm_K(k)\times \delta_{a_s}\times   \chi_{w_s^\prime V}(k^\prime)dm_K(k^\prime) \sinh(s) ds\,.$$
  To check this identity, it suffices to test it against product functions of the form $a_1(k)b(a_s)a_2(k^\prime)$, where it is follows immediately from the definition of $J$.  
   
    Let $\cM$ be the multiplication map into $G$, so that 
   $$\cM(k,n_t,k^\prime):=kn_tk^\prime=kw_s a_s w_s^\prime k^\prime= \cM\circ J(k, n_t, k^\prime) $$
    namely $\cM=\cM\circ J$.   For any measure $\nu$ on $K\times N_+\times K$ we have $\cM_\ast(\nu)= \left(\cM\circ J\right)_\ast(\nu)$. It follows from the explicit expression for $J_\ast(\nu)$ above that for $f\in C_c(G)$ :
   $$\int_K\int_K \int_{n_t\in N_+} f(kn_t k^\prime)t\,dt\,\chi_U(k)dm_K(k) \chi_V(k^\prime)dm_K(k^\prime) $$
$$=   \int_K\int_K \int_{a_s \in A_+} f(kw_s a_s w_s^\prime k^\prime)\sinh(s)\,ds\,\chi_U(k)dm_K(k) \chi_V(k^\prime)dm_K(k^\prime)\,. $$

%
 
   \end{rem}   
 
%
   
  {\it Proof of Theorem \ref{lem:open}}.  
Fix $\epsilon > 0$. As noted above, for $r,t>0$ 
 there are unique elements $w_r,w^\prime_r \in K$ such that $n_t = w_r a_r w^\prime_r$, with $t=2\sinh r/2$.  Define $U_r=\cup_{r \le s< r+\epsilon } Uw_s^{-1}$ and 
  $V_r= \cup_{r \le s <  r+\epsilon} (w_s^\prime)^{-1}V$. Let $\nu_r$ be the normalized restriction of $m_K$ to $U_r$ and $\lambda_r$ be the normalized restriction of $m_K$ to $V_r$. We will show that there is a constant $C_r>1$ such that $\lim_{r\to\infty} C_r = 1$ and 
 $$\sigma^{U,V}_{r,\epsilon} \le C_r \nu_r\ast \eta_{r,\epsilon}\ast \lambda_r\,,$$
 where $ \eta_{r,\epsilon}$ is given in Definition  \ref{eta-r-e} above. 
 We prove the above inequality by comparing Radon-Nikodym derivatives of the two measures in question, for each given $r$.  Using the formula for Haar measure on $G$ in polar coordinates, we have for $f\in C_c(G)$
 

$$\sigma^{U,V}_{r,\epsilon} (f)=\int_{k\in K}  \int_{k^\prime\in K} \int_{s\in [r, r+\epsilon)} f(ka_s k^\prime)\frac{\sinh s ds}{\cosh (r+\epsilon)-\cosh r} \frac{\chi_U(k)dm_K(k)}{m_K(U)}\frac{\chi_{V}(k^\prime)dm_K(k^\prime)}{m_K(V)}.$$

On the other hand by definition of convolution 
\begin{eqnarray*}
&&\nu_r\ast \eta_{r,\epsilon}\ast \lambda_r(f)\\
&=&\int_{k\in K} \int_{k^\prime\in K} \int_{2\sinh(r/2)}^{2\sinh((r+\epsilon)/2)} f(kn_t k^\prime)\frac{ tdt}{\cosh (r+\epsilon) -\cosh r}\frac{\chi_{U_r}(k)dm_K(k)}{m_K(U_r)}\frac{\chi_{V_r}(k^\prime)dm_K(k^\prime)}{m_K(V_r)}\,,
\end{eqnarray*}
and using Remark \ref{integral} 
$$=\int_{k\in K} \int_{k^\prime\in K} \int_{s\in [r,r+\epsilon)}  f(kw_s a_s w_s^\prime k^\prime)\frac{ \sinh s ds}{\cosh (r+\epsilon) -\cosh r}\frac{\chi_{U_r}(k)dm_K(k)}{m_K(U_r)}\frac{\chi_{V_r}(k^\prime)dm_K(k^\prime)}{m_K(V_r)}\,.$$

Note that the support of $\sigma^{U,V}_{r,\epsilon}$ is contained in the support 
of the convolution above, by definition of $U_r$ and $V_r$. Furthermore
$$\frac{d\sigma^{U,V}_{r,\epsilon}}{d\left(\nu_r \ast\eta_{r,\epsilon}\ast \lambda_r\right) }(g)=\frac{m_K(U_r)m_K(V_r)}{m_K(U)m_K(V)}=:C_r\,,$$
and since $w_r \to 1$ and $w^\prime_r$ tends to the $180^o$ rotation as $r \to \infty$ (by Lemma \ref{NtoA}), it follows that $C_r \to 1$ as $r \to \infty$. Indeed, since $U$ is compact and $s \mapsto w_s$ is continuous, the set
$$U'_r:=\cup_{r \le s\le r+\epsilon } Uw_s^{-1}w_r$$
is compact, $m_K(U) \le m_K(U_r) \le m_K(U'_r)$. Moreover, $U \subset U'_r$ and $U'_r$ is contained in the $\delta(r)$-neighborhood of $U$ for some $\delta(r)>0$ satisfying $\lim_{r\to\infty} \delta(r) = 0$ (by Lemma \ref{NtoA}). Since the intersection of these neighborhoods is $U$, it follows that $m_K(U_r) \to m_K(U)$ as $r\to\infty$. Similarly, $m_K(V_r) \to m_K(V)$ as $r\to\infty$.



To complete the proof of the Theorem \ref{lem:open} it suffices, by Lemma \ref{lem:strong-domination} (setting $\tau_r=\sigma^{U,V}_{r,\epsilon}$ and $\tau^\prime_r= \nu_r \ast\eta_{r,\epsilon}\ast \lambda_r$) to establish the conclusions for $\nu_r \ast\eta_{r,\epsilon}\ast \lambda_r$. By Proposition \ref{SL2 : general averages},  $m_K\ast \eta_{r,\epsilon}\ast m_K$ is a good averaging family. Since for all $r > 1$ 
$$\nu_r \ast\eta_{r,\epsilon}\ast \lambda_r\le \frac{1}{m_K(U_r)m_K(V_r)}m_K\ast \eta_{r,\epsilon}\ast m_K  \le \frac{C}{m_K(U)m_K(V)}m_K\ast \eta_{r,\epsilon}\ast m_K 
$$
for some $C>0$, the Domination Lemma \ref{lem:domination} implies  $r\mapsto \nu_r \ast\eta_{r,\epsilon}\ast \lambda_r$ satisfies the strong type $(p,p)$, $1 < p < \infty$ and $L\log L$ maximal inequalities. 

Let $\nu$ denote the normalized restriction of $m_K$ to $U$ and $\lambda$ denote the normalized restriction of $m_K$ to $V$. Then $\frac{d\nu_r}{dm_K}\to\frac{d\nu}{dm_K}$, $\frac{d\lambda_r}{dm_K}\to \frac{d\lambda}{dm_K}$ in $L^1(K)$ norm. So Proposition \ref{SL2 : general averages} implies $r \mapsto \nu_r \ast\eta_{r,\epsilon}\ast \lambda_r$ is a good averaging family. \qed

 We now pass from $\sigma^{U,V}_{r,\epsilon}$ to averages defined by arbitrary densities on $K$, as follows. 
     \begin{lem}\label{lem:strong-domination2}
  Suppose $\{\tau_r\}_{r>0}$ and $\{\tau^\prime_{n,r}\}_{n \in \N,r>0}$ are families of probability measures on $G$ and $C_{n}>1$ satisfies:
  \begin{itemize}
  \item $\tau_r \le C_{n} \tau^\prime_{n,r}$ for all $r,n$;
  \item $C_n \to 1$ as $n \to \infty$;
  \item for each $n\in \N$, $\{\tau^\prime_{n,r}\}_{r>0}$ is a good averaging family.
  \end{itemize}
   Then $\{\tau_r\}_{r>0}$ is also a good averaging family.
  \end{lem}
 
 \begin{proof}
 It follows from the Domination Lemma \ref{lem:domination} that $\{\tau_r\}_{r>0}$ satisfies the strong type $L\log L$ maximal inequality and the strong type $(p,p)$ maximal inequalities for $1<p<\infty$. Let $f\in L^\infty(X,\mu)$ be nonnegative. Then $\tau^\prime_{n,r}(f)$ converges to $\E[f|G]$ pointwise a.e. as $r \to \infty$. Since $\tau_r(f) \le C_n \tau^\prime_{n,r}(f)$ it follows that $\limsup_{r\to\infty} \tau_r(f)(x) \le \limsup_{n\to \infty}C_n\cdot \E[f|G](x)$ for a.e. $x$, and  since $C_n \to 1$ as $n\to\infty$ we have $\limsup_{r\to\infty} \tau_r(f)(x) \le \E[f|G](x)$ for a.e. $x$. The proof is now identical to the end of the proof of Lemma \ref{lem:strong-domination}. 
   \end{proof}

\begin{proof}[Proof of Theorem \ref{thm:sector1}]
 Let $A, B \subset K$ be open sets whose complements $U:=K-A, V:=K-B$ have positive measure. By Theorem \ref{lem:open}, $\{\sigma^{U,K}_{r,\epsilon}\}_{r>0}$, $\{\sigma^{K,V}_{r,\epsilon}\}_{r>0}$ and $\{\sigma^{U,V}_{r,\epsilon}\}_{r>0}$ are good averaging families. Since
$$\sigma^{A,B}_{r,\epsilon} = \frac{\sigma_{r,\epsilon} - m_K(U)\sigma^{U,K}_{r,\epsilon} - m_K(V)\sigma^{K,V}_{r,\epsilon}+m_K(U)m_K(V) \sigma^{U,V}_{r,\epsilon}}{1-m_K(U)-m_K(V) + m_K(U)m_K(V)}$$
it follows that $\{\sigma^{A,B}_{r,\epsilon}\}_{r>0}$ is also a good averaging family.

Now let $A, B \subset K$ be Borel sets with positive measure. We will show that $\{\sigma^{A,B}_{r,\epsilon}\}_{r>0}$ is a good averaging family. For each $n>0$ there  exist open sets $U_n\supset A$ and $V_n \supset B$ such that $m_K(U_n\setminus A)<1/n$ and $m_K(V \setminus B)<1/n$. By Lemma \ref{lem:open} $\{\sigma^{U_n,V_n}_{r,\epsilon}\}_{r>0}$ is a good averaging family. Since
$$\sigma^{A,B}_{r,\epsilon} \le \frac{m_K(U_n)m_K(V_n)}{m_K(A)m_K(B)} \sigma^{U_n,V_n}_{r,\epsilon}$$
it follows from Lemma \ref{lem:strong-domination2} that $\{\sigma^{A,B}_{r,\epsilon}\}_{r>0}$ is a good averaging family.

Let now $\nu$ and $\lambda$ be arbitrary probability measures on $K_0$ with bounded densities, namely $\frac{d\nu}{dm_K},\frac{d\lambda}{dm_K} \in L^\infty(K)$.  Recall that a simple function is a finite linear combination of characteristic functions of Borel subsets. Since $\frac{d\nu}{dm_K},\frac{d\lambda}{dm_K}$ are essentially bounded, for any $n\in \N$ there exist simple functions $y_{\nu,n}\,,\, y_{\lambda,n} \in L^\infty(K,m_K)$ such that $y_{\nu,n}\ge \frac{d\nu}{dm_K}$,  $y_{\lambda,n}\ge 
\frac{d\lambda}{dm_K}$  and $\|y_{\nu,n}-\frac{d\nu}{dm_K} \|_\infty \le 1/n$, $\|y_{\lambda,n}-\frac{d\lambda}{dm_K}\|_\infty \le 1/n$.
 Then $y_{\nu,n}/\norm{y_{\nu,n}}_1$ and $y_{\lambda,n}/\norm{y_{\lambda,n}}_1$ are probability densities and simple functions. Denoting the probabilities they define by $\nu_n$ and $\lambda_n$, clearly $\nu\le (1+1/n)\nu_n$ and $\lambda \le (1+1/n)\lambda$.  Because $y_{\nu,n},y_{\lambda,n}$ are simple it follows from the previous paragraph and linearity that $\{\nu_n \ast \alpha_{r,\epsilon} \ast \lambda_n\}_{r>0}$ is a good averaging family for each $n$.
  Since $\nu \le (1+1/n)\nu_n, \lambda \le (1+1/n)\lambda_n$, it follows that 
$$\nu \ast \alpha_{r,\epsilon} \ast \lambda \le (1+1/n)^2 \nu_n \ast \alpha_{r,\epsilon} \ast \lambda_n.$$
So Lemma \ref{lem:strong-domination2}  implies $\{\nu \ast \alpha_{r,\epsilon} \ast \lambda\}_{r>0}$ is a good averaging family.

\end{proof}  

%
%

\section{Ergodic theorems for general real rank one groups}

\subsection{Structure theory for real rank one groups}
In the present section we will extend Theorem \ref{thm:main} to general real-rank one groups using the method of rotations, applied to totally geodesic embeddings. We assume that $G$ is a real-rank one connected non-compact simple Lie group with finite center. In the present section our notation will be different from the notation used thus far, where $K$, $A$, and $N$ denoted specific subgroups of $SL_2(\R)$. We now fix a maximal compact subgroup of $G$ and denoted it by $K$, and a one-parameter subgroup $A\cong \R$ of $G$ such that $G=KA K$ is a Cartan decomposition. We let $N$ be the horospherical subgroup of $G$ associated with $A$, so that $G=KAN$ is an Iwasawa decomposition.

 Let $\fg$ denote the Lie algebra of $G$.  Fix a Cartan involution $\theta$ on $G$ and $\fg$, and let $\fg=\fk\oplus \fp$ be the associated Cartan decomposition of $\fg$ to the $\pm 1$ eigenspaces of $\theta$. Choose a maximal Abelian subalgebra $\fa$ contained in $\fp$. Because $G$ has real rank 1, $\dim_\R \fa=1$. Let $\fa^\ast=\Hom (\fa ,\R)$ denote the real dual of $\fa$, and let $\Sigma=\Sigma(\fa,\fg)\subset \fa^\ast$ denote the set of non-zero roots of $\fa$ in $\fg$. Because $G$ has real rank one, $\Sigma=\set{\pm \alpha}$ for some $\alpha \in \fa^\ast$, or $\Sigma=\set{\pm \alpha,\pm 2\alpha}$. The Weyl group $W=W(\fa,\fg)$ is isomorphic to $\Z_2$ in both cases, and its nontrivial element acts as multiplication by $-1$ on $\fa$.  The adjoint action of the Lie algebra $\fa$ on $\fg$ is diagonalizable, with the eigenspaces being $\fg_{\pm\alpha}$, $\fg_{\pm 2\alpha}$ (when non-empty), and $\fg_0$. $\fg$ is the direct sum of these subspaces, and $\fg_0=\fa\oplus \fm$, where $\fm$ is the centralizer of $\fa$ in $\fg$. Denote  $m_1=\dim_\R \fg_\alpha$, $m_2=\dim_\R \fg_{2\alpha}$. 
We fix an element $H_1\in \mathfrak{a}$, satisfying $\alpha(H_1)=1$, so that  
 $\set{e^{t H_1}}_{t\in \R}$ is a parametrization of $A$. 


 
\begin{lem}[KAK decomposition]\label{lem:KAK}
Let $m_K$ denote the Haar measure on $K$ normalized to have total mass one. Let $m_1,m_2\ge 0$ be as above and  let $m_G$ denote the measure on $G$ defined by
$$\int F(g)~dm_G(g) = \int_K \int_0^\infty \int_K F(k_1e^{tH_1}   k_2) \sinh(t)^{m_1+m_2} \cosh(t)^{m_2}~dm_K(k_1) dt dm_K(k_2).$$
Then $m_G$ is a Haar measure on $G$.
\end{lem}
\begin{proof}
For this well-known formula, see e.g. \cite{He2} or \cite[Eqs. (2.5), (4.8)]{Koo84}. 
\end{proof}


We now turn to choose the subgroup $L\subset G$ to which we will apply the method of rotations, using \cite[Prop. 6.52, p. 321]{Kn} in our discussion. If $\mathfrak{g}_{2\alpha}=0$, let $X_\alpha\in \mathfrak{g}_\alpha$ be any non-zero vector, and let $\mathfrak{l}$ be the Lie algebra spanned by $X$, $Y=\theta(X)$ and $H=[X, Y]$. Then $\mathfrak{l}$ is a Lie algebra isomorphic to $\mathfrak{s}\mathfrak{l}_2(\R)$, and it is invariant under $\theta$. The restriction of $\theta$ to $\mathfrak{l}$ is a Cartan involution of $\mathfrak{l}$, and $\mathfrak{a}$ is contained in $\mathfrak{l}$ and spanned by $H$. Multiplying $X$ by a suitable multiple if necessary, we can assume that the map 
$E_{1,2}\mapsto X$, $E_{2,1}\mapsto Y$, $\text{diag}(1/2,-1/2)\mapsto H_1$ is a Lie algebra isomorphism $\tau:  \mathfrak{s}\mathfrak{l}_2(\R)\to \mathfrak{l}$. Here $E_{i,j}$ is the elementary $2\times 2$ matrix with $1$ at the $(i,j)$ place. 

If $\mathfrak{g}_{2\alpha}\neq 0$, we choose any non-zero $X\in \mathfrak{g}_{2\alpha}$, 
and consider the Lie algebra $\mathfrak{l}$ spanned by $X$, $Y=\theta(X)$ and $H=[X,Y]$. Again $\mathfrak{l}$ is isomorphic with $\mathfrak{s}\mathfrak{l}_2(\R)$ and contains $\mathfrak{a}$. Note however that the element $H_1\in \mathfrak{a}$ we chose above to parametrize $A$ now has the following property.  When viewed as an element of the $\R$-split Cartan subalgebra $\mathfrak{a}$ of $\mathfrak{l}$, 
the evaluation of the unique root of $\mathfrak{a}$ (in $\mathfrak{l}$) on $H_1$ gives the value $2$, and not $1$. 
Thus, multiplying $X$ by a suitable multiple if necessary, we can assume that the Lie algebra isomorphism $\tau :\mathfrak{s}\mathfrak{l}_2(\R)\to \mathfrak{l}$ is given now by $E_{1,2}\mapsto X$, $E_{2,1}\mapsto Y$, $\text{diag}(1,-1)\mapsto H_1$.

We let $L$ denote the closed subgroup of $G$ with Lie algebra $\mathfrak{l}$, and then $L$ is isomorphic to a finite covering group of $PSL_2(\R)$.  We denote $K_L:=K\cap L, N_L:=N \cap L$, and then $L = K_LAN_L$ is an Iwasawa decomposition of $L$, and $L=K_L A K_L$ is a Cartan decomposition.

 The restriction of (any multiple of) the Killing form on the Lie algebra $\mathfrak{g}$ to the Lie algebra $\mathfrak{l}$ is a non-degenerate invariant form, and hence a multiple of the Killing form on $\mathfrak{l}$. Pulling back this form to $\mathfrak{s}\mathfrak{l}_2(\R)$ via the representation $\tau$, we obtain a multiple of the Killing form on $\mathfrak{s}\mathfrak{l}_2(\R)$, and upon restriction also a multiple $c\cG$ of the Riemannian metric $\cG$ on $\H^2$ used in \S 4. In the first case, when $\mathfrak{g}_{2\alpha}=0$, the multiple is clearly $c=1$, and in the second case, when $\mathfrak{g}_{2\alpha}\neq 0$, the multiple is clearly $c=1/2$.

 Consider now the case where $G$ is an adjoint group, namely it is the unique group with trivial center in the class of groups with isomorphic Lie algebras. It then follows that $L$ is in fact isomorphic to $PSL_2(\R)$ itself. This fact can be verified directly using the explicit formulas for the action of the isometry groups of hyperbolic spaces stated in \cite[Chapter II.10.25]{BH}. 
   
\begin{lem}\label{lem:convert}
Let $\tau : PSL_2(\R)\to L\subset G$ be the representation constructed above, with $G$ an adjoint group. Let $N_L=\{n_t^\tau=\tau(n_t)\}_{t\in \R}$, $A_L=\set{a_r^\tau=\tau(a_r)}_{r\in \R}=A$ where $n_t, a_r\in PSL_2(\R)$ are the parametrizations indicated in \S 4.1.  There exists a positive constant $c=c_\tau$ such that for all $t,r>0$ the following are equivalent:
\begin{itemize}
\item[1.] $K_L n^\tau_t K_L = K_L a^\tau_{r/c}K_L$,
\item[2.] $t=2\sinh(r/2c)$,
\item[3.] $\cosh(r/c) = 1 + t^2/2$,  
\item[4.] $Kn^\tau_tK = Ka^\tau_{r/c}K$.
\end{itemize}
\end{lem}

\begin{proof}
Let $p_0\in G/K$ be the unique point in the symmetric space $G/K$ with stability group $K$, so that the stability group of $p_0$ in $L$ is $K_L=K\cap L$. We have, by definition 
$d_{G/K}(\tau(y)p_0,p_0)=d_c(y\cdot o, o)=cd(y\cdot o,o)$ for all $y\in SL_2(\R)$, where $d_{G/K}$ is the invariant metric on $G/K$, $d$ the metric on $\H^2_{-1}$ associated with constant curvature $-1$, and $o$ a suitable reference point. The fact that 2 and 3 are equivalent to $K_L n^\tau_tK_L = K_L a^\tau_{r/c}K_L $ follows immediately from our discussion of the $ \text{PSL}_2(\R)$ case in Lemma \ref{distance on horocycle} and Remark \ref{curvature}. It remains to show that  
$K_L n^\tau_tK_L = K_L a^\tau_{r/c}K_L $ follows from $K n^\tau_tK = K a^\tau_{r/c}K $.  
It is well-known (see e.g. \cite{He2}) that the radial component of the Cartan decomposition in the real rank one group $G$ is determined uniquely. This is equivalent to the fact that in the symmetric space $G/K$ we have $d_{G/K}(g p_0,p_0)=d_{G/K}(h p_0,p_0)$ if and only if $KgK=KhK$. Thus if  
$K n^\tau_tK = K a^\tau_{r/c}K$ then $d_{G/K}(n_t^\tau p_0,p_0)=d_{G/K}(a^\tau_{r/c}p_0,p_0)$. The distance $d_{G/K}$ restricts to a distance on the totally geodesic hyperbolic plane $L\cdot p_0\cong L/K_L$. Using the parametrization of this plane via the representation $\tau$ of $SL_2(\R)$, by Remark \ref{curvature} it follows that $K_L n^\tau_tK_L = K_L a^\tau_{r/c}K_L $. 
\end{proof}

Thus $A\cong \R$ is parametrized by $\set{a_r^\tau}_{r\in \R}$, and also by  $\set{e^{r H_1}}_{r\in \R}$. These parametrizations are identical when $\mathfrak{g}_{2\alpha}=0$, but otherwise they are different and satisfy $\tau(a_{2t})=a_{2t}^\tau=e^{tH_1}$.  

\begin{lem}[$KN_LK$ decomposition]\label{lem:KNK2}
Let $G$ be a connected simple adjoint Lie group of real rank one and finite center, and $L\subset G$ 
chosen as above. Then $G=KN_LK$, and there exists a function $\psi$ on $[0,\infty)$ satisfying, for any bounded measurable function $F$ on $G$ with compact support 
$$\int_G F(g)~dm_G(g) = \int_K \int_0^\infty \int_K F(k_1 n^\tau_t k_2) \psi(t)~dm_K(k_1) dt dm_K(k_2), $$
and $\psi$ has the following asymptotic form :
\begin{enumerate}
\item when $m_2 > 0$, namely when $\mathfrak{g}_{2\alpha}\neq 0$, 
$$\psi(T) = C_G T^{m_1 + 2m_2-1} + O(T^{m_1 + 2m_2-2}) \quad \textrm{for } T\ge 1$$

\item when $m_2=0$, namely when $\mathfrak{g}\cong \mathfrak{s}\mathfrak{o}(m_1+1,1)$ 
$$\psi(T)=C_GT^{2m_1-1}+O\left(T^{2m_1-2}\right)\quad \textrm{for } T \ge 1.$$

\end{enumerate}
\end{lem}

\begin{proof}
Let us consider first the case where $\mathfrak{g}_{2\alpha}\neq 0$. Define $\psi$ by
$$\psi(T) = \frac{\sinh^{m_1+m_2}(R)\cosh^{m_2}(R)}{2\cosh(R)}$$
where $T=2\sinh(R)$. $\psi$  is well-defined since $R \mapsto 2\sinh(R)$ is invertible on $[0,\infty)$. 
Now since 
$$\sinh^{m_1+m_2}(R)\cosh^{m_2}(R) = 2\psi(T) \cosh(R)$$
we can conclude 
$$\int_0^{R} \sinh^{m_1+m_2}(r)\cosh^{m_2}(r)~dr = \int_0^{T(R)} \psi(t)~dt$$
upon differentiating both sides with respect to $R$, and using 
$\frac{dT(R)}{dR}=2\cosh(R)$. 

Suppose $\chi_{B_R}$ is the characteristic function of a ball of radius $R$ in $G/K$ with center $p_0$. Then by Lemmas \ref{lem:KAK} and \ref{lem:convert}, since $d_{G/K}(n_t^\tau p_0,p_0)=r\iff  t=2\sinh (r)$: 
$$\int_{G} \chi_{B_{R}}(g)~dm_G(g) =\int_K\int_0^R\int_K \chi_{B_R}(k_1 a^\tau_{2r} k_2) \sinh^{m_1+m_2}(r)\cosh^{m_2}(r)dk_1dr dk_2$$
$$= \int_K \int_0^{T(R)} \int_K \chi_{B_{R}}(k_1 n_t^\tau k_2) \psi(t)~dm_K(k_1) dt dm_K(k_2)$$
$$=\int_K \int_0^{\infty} \int_K \chi_{B_{R}}(k_1 n_t^\tau k_2) \psi(t)~dm_K(k_1) dt dm_K(k_2)
$$
 Therefore this formula holds for all radial functions. Because the measure on the right-hand-side is bi-$K$-invariant, this formula must hold for all bounded measurable functions with compact support.

The formula $G=KN_LK$ is immediate from Lemma \ref{lem:convert}. It remains to prove the asymptotic formula for $\psi$. Clearly $\sinh^2(R) = \cosh^2(R)-1 = T^2/4$ 
implies $\cosh(R) = \sqrt{T^2/4+1}$, so we obtain
\begin{eqnarray*}
\psi(T) &=& 2^{-(m_1+ m_2+1)}T^{m_1+m_2}(T^2/4+1)^{(m_2-1)/2}\\
&=& C^\prime_G T^{m_1 + 2m_2-1} + O(T^{m_1 + 2m_2-2}) \quad \textrm{for }T\ge 1.
\end{eqnarray*}
where $C_G>0$ is a constant depending only on $G$.

The case where $\mathfrak{g}_{2\alpha}=0$ is handled similarly, defining 
$\psi(T) = \frac{\sinh^{m_1}(R)}{\cosh(R/2)}$, with $T=2\sinh R/2$.
Then $\int_0^R \sinh^{m_1}r dr=\int_0^{T(R)} \psi(t)dt$, and using 
$\sinh R=2\sinh R/2\cosh R/2$ we have $\psi(T)=2^{m_1}T^{m_1}(\sqrt{T^2/4+1})^{m_1-1}$ so that $\psi(T)=C_GT^{2m_1-1}+O\left(T^{2m_1-2}\right)$ for $T\ge 1$.
\end{proof}

We note that for a finite cover $\tau^\prime : L \to PSL_2(\R)$ the kernel is central, and so the inverse image of the subgroups $A$ and $N$ of $PSL_2(\R)$, denoted $a^{\tau^\prime}_r$ and $n_t^{\tau^\prime}$ are isomorphic to $A$ and $N$. Using these subgroups of $L$ in the foregoing argument, we see that Theorem \ref{lem:KNK2} holds for any finite cover group, not just the adjoint group. 
\subsection{Proof of the ergodic theorems for real rank one groups}

Let us prove the ergodic theorems for averages on real rank one groups, starting with the radial case.  We will start by assuming $G$ is the adjoint group, and complete the argument for the general case at the end of the section. 
\begin{thm}\label{first-main} Let $G$ be a connected simple Lie group of real rank and finite center, and fix any invariant Riemannian metric on the symmetric space $G/K$. Then 
$\sigma_{r,\epsilon}$ and $\beta_r$ are good averaging families, for every fixed $\epsilon > 0$. 
\end{thm}

\begin{proof}
When $G$ is the adjoint group, the proof is virtually the same as the proof of Theorem \ref{thm:sl2} when $\mathfrak{g}_{2\alpha}=0$. For completeness we provide the details in the case $\mathfrak{g}_{2\alpha} \neq 0$. Let $\eta$ be the measure on $N_L$ defined by
$$\eta(E) = \int 1_E(n^\tau_t)\psi(t) dt$$
where $\psi$ is the function defined in Lemma \ref{lem:KNK2} and $n^\tau_t$ is as in Lemma \ref{lem:convert}.  Let $\eta_{R,\epsilon}$ be the measure on $N_L$ defined by
$$\eta_{R,\epsilon}(E) =  \frac{\eta(E \cap \{n_t^\tau:~ t\in [2\sinh(R),2\sinh((R+\epsilon))]\} )}{\eta(\{n_t^\tau:~ t\in [2\sinh(R),2\sinh((R+\epsilon))]\} )}.$$
Theorem \ref{thm:R} implies $\{\eta_{R,\epsilon}\}_{R>0}$ is an $L^1$-good averaging family for $N_L$. By the Howe-Moore Theorem, $\{\eta_{R,\epsilon}\}_{R>0}$ is an $L^1$-averaging family for $G$. 

By Lemma \ref{lem:KNK2}, $m_K*\eta_{R,\epsilon}*m_K = \sigma_{R,\epsilon}$. So Proposition \ref{max-conv}  now implies $\{\sigma_{R,\epsilon}\}_{R>0}$ satisfies the strong $(p,p)$ type maximal inequality $(p>1$) and the $L\log L$ maximal inequality. The bounded convergence theorem implies that $\{\sigma_{R,\epsilon}\}_{R>0}$ is pointwise ergodic in $L^\infty$. So Theorem \ref{thm:dense} implies $\{\sigma_{R,\epsilon}\}_{R>0}$ is pointwise ergodic in $L^p$ for all $p>1$ and in $L\log L$. The case of $\{\beta_r\}_{r>0}$ is handled similarly. For the case where $G$ has finite center, see below. 
\end{proof}

We now turn to the proof of Theorem \ref{thm:main} for non-radial averages.  
For $r,\epsilon>0$, let $\alpha_{r,\epsilon}$ denote the probability measure on $A\subset G$ given by
$$ \alpha_{r,\epsilon} = \frac{\int_r^{r+\epsilon} \sinh(t)^{m_1+m_2} \cosh(t)^{m_2} \delta_{e^{tH_1}}~dt}{\int_r^{r+\epsilon} \sinh(t)^{m_1+m_2} \cosh(t)^{m_2}~dt}.$$
For example, note that $m_K\ast \alpha_{r,\epsilon} \ast m_K = \sigma_{r,\epsilon}$ where $m_K$ denotes Haar probability measure on $K$. Recall the definition of $\sigma^{U,V}_{r,\epsilon}$ from \S \ref{sec:stmt}. We first prove a special case of Theorem \ref{thm:main}:

\begin{thm}\label{thm:open2}
For any compact $Z$-invariant subsets $U,V \subset K$ both with positive measure, the families $\{\sigma^{U,V}_{r,\epsilon}\}_{r>0}$ and $\{\beta^{U,V}_r\}_{r>0}$ are both good averaging families.
\end{thm}  

\begin{proof}
We first assume that $G$ is an adjoint group. The proof then is very similar to the proof of Theorem \ref{lem:open} when $\mathfrak{g}_{2\alpha}=0$ , and we provide the details when $\mathfrak{g}_{2\alpha}\neq 0$. If $t,r>0$ are such that $Kn_t^\tau K = Ka_{2r}^\tau K$ then let $w_r,w^\prime_r \in K \cap L$ be the elements satisfying $n^\tau _t=w_r a_{2r}^\tau w^\prime_r$.  Note the important fact that these identities hold in the subgroup $L\cong PSL_2(\R)$, and hence $a_{2r}$, $w_r$ and $w_r^\prime$ are unique and form continuous functions of $t$.  
Define $U_r=\cup_{r \le s< r+\epsilon } Uw_s^{-1}$ and 
  $V_r= \cup_{r \le s <  r+\epsilon} (w_s^\prime)^{-1}V$. Let $\nu_r$ be the normalized restriction of $m_K$ to $U_r$ and $\lambda_r$ be the normalized restriction of $m_K$ to $V_r$. We will show that there is a constant $C_r>1$ such that $\lim_{r\to\infty} C_r = 1$ and 
 $$\sigma^{U,V}_{r,\epsilon} \le C_r \nu_r\ast \eta_{r,\epsilon}\ast \lambda_r$$
 where $\eta_{r,\epsilon}$ is as in the proof of Theorem \ref{first-main}.
 
 Using the formula for Haar measure on $G$ in polar coordinates, we have  for any bounded measurable function $f$ on $G$
$$\sigma^{U,V}_{r,\epsilon} (f)=\int_{k\in K} \int_{k^\prime\in K} \int_{r}^{r+\epsilon} f(k e^{sH_1} k^\prime)\frac{\sinh(s)^{m_1+m_2} \cosh(s)^{m_2} ds}{\int_{r}^{r+\epsilon} \sinh(s)^{m_1+m_2} \cosh(s)^{m_2} ds} \frac{\chi_U(k)dm_K(k)}{m_K(U)}\frac{\chi_{V}(k^\prime)dm_K(k^\prime)}{m_K(V)}.$$

On the other hand by definition of convolution 
\begin{eqnarray*}
&&\nu_r\ast \eta_{r,\epsilon}\ast \lambda_r(f)\\
&=&\int_{k\in K}  \int_{k^\prime\in K} \int_{2\sinh(r)}^{2\sinh((r+\epsilon))} f(kn^\tau_t k^\prime)\frac{ \psi(t)~dt}{\int_{2\sinh(r)}^{2\sinh((r+\epsilon))} \psi(t)~dt}\frac{\chi_{U_r}(k)dm_K(k)}{m_K(U_r)}\frac{\chi_{V_r}(k^\prime)dm_K(k^\prime)}{m_K(V_r)}\,,
\end{eqnarray*}
and using Lemma \ref{lem:KNK2} 
$$=\int_{k\in K} \int_{k^\prime\in K}  \int_{r}^{r+\epsilon} f(kw_s a^\tau_{2s} w_s^\prime k^\prime)\frac{\sinh(s)^{m_1+m_2} \cosh(s)^{m_2} ds}{ \int_{r}^{r+\epsilon} \sinh(s)^{m_1+m_2} \cosh(s)^{m_2} ds} \frac{\chi_{U_r}(k)dm_K(k)}{m_K(U_r)}\frac{\chi_{V_r}(k^\prime)dm_K(k^\prime)}{m_K(V_r)}\,.$$

Note that the support of $\sigma^{U,V}_{r,\epsilon}$ is contained in the support 
of the convolution above, by definition of $U_r$ and $V_r$. Furthermore
$$\frac{d\sigma^{U,V}_{r,\epsilon}}{d\left(\nu_r \ast\eta_{r,\epsilon}\ast \lambda_r\right) }(g)=\frac{m_K(U_r)m_K(V_r)}{m_K(U)m_K(V)}=:C_r\,,$$
and since $w_r \to 1$ and $w^\prime_r$ tends to the $180^o$ rotation as $r \to \infty$ (by Lemma \ref{NtoA}), it follows that $C_r \to 1$ as $r \to \infty$. Indeed, since $U$ is compact and $s \mapsto w_s$ is continuous, the set
$$U'_r:=\cup_{r \le s\le r+\epsilon } Uw_s^{-1}w_r$$
is compact, $m_K(U) \le m_K(U_r) \le m_K(U'_r)$. Moreover, $U \subset U'_r$ and $U'_r$ is contained in the $\delta(r)$-neighborhood of $U$ for some $\delta(r)>0$ satisfying $\lim_{r\to\infty} \delta(r) = 0$ (by Lemma \ref{NtoA}). Since the intersection of these neighborhoods is $U$, it follows that $m_K(U_r) \to m_K(U)$ as $r\to\infty$. Similarly, $m_K(V_r) \to m_K(V)$ as $r\to\infty$.



To complete the proof it suffices, by Lemma \ref{lem:strong-domination} (setting the averages $\tau_r$ and $\tau_r^\prime$ that appear there as $\tau_r=\sigma^{U,V}_{r,\epsilon}$ and $\tau^\prime_r= \nu_r \ast\eta_{r,\epsilon}\ast \lambda_r$) to establish the conclusions for $\nu_r \ast\eta_{r,\epsilon}\ast \lambda_r$. By Proposition \ref{SL2 : general averages} (which holds for general real rank 1 groups with $\eta_{r,\epsilon}$ as in the proof of Theorem \ref{first-main} by exactly the same argument),  $\{m_K\ast \eta_{r,\epsilon}\ast m_K\}_{r>0}$ is a good averaging family. Since for all $r > 1$ 
$$\nu_r \ast\eta_{r,\epsilon}\ast \lambda_r\le \frac{1}{m_K(U_r)m_K(V_r)}m_K\ast \eta_{r,\epsilon}\ast m_K  \le \frac{C}{m_K(U)m_K(V)}m_K\ast \eta_{r,\epsilon}\ast m_K 
$$
for some $C>0$, the Domination Lemma \ref{lem:domination} implies  $r\mapsto \nu_r \ast\eta_{r,\epsilon}\ast \lambda_r$ satisfies the strong type $(p,p)$, $1 < p < \infty$ and $L\log L$ maximal inequalities. 

Let $\nu$ denote the normalized restriction of $m_K$ to $U$ and $\lambda$ denote the normalized restriction of $m_K$ to $V$. Then $\frac{d\nu_r}{dm_K}\to\frac{d\nu}{dm_K}$, $\frac{d\lambda_r}{dm_K}\to \frac{d\lambda}{dm_K}$ in $L^1(K)$ norm. So Proposition \ref{SL2 : general averages} implies $r \mapsto \nu_r \ast\eta_{r,\epsilon}\ast \lambda_r$ is a good averaging family. 
\end{proof}

\begin{thm}\label{thm:sector2}
As above, let $G$ be a connected non-compact simple Lie group of real rank one with finite center. Fix a maximal compact subgroup $K$. If $\nu,\lambda<<m_K$ are $Z$-invariant probability measures with densities $\frac{d\nu}{dm_K},\frac{d\lambda}{dm_K} \in L^\infty(K,m_K)$ and $\epsilon>0$ then $\{\nu\ast \alpha_{r,\epsilon}\ast \lambda\}_{r>0}$
is a good averaging family. 
\end{thm}

\begin{proof}
When $G$ is adjoint, he proof is essentially the same as the proof of Theorem \ref{thm:sector1} using $\{n^\tau_t\}$ in place of $\{n_t\}$.
\end{proof}

{\it Proof of Theorem \ref{thm:main}}.  For adjoint groups Theorem \ref{thm:main} follows immediately from Theorem \ref{thm:sector2} by setting $\nu=m_K(U)^{-1}\chi_U$ and $\lambda=m_K(V)^{-1}\chi_V$. Indeed then $\nu\ast \alpha_{r,\epsilon} \ast \lambda= \sigma^{U,V}_{r,\epsilon}$. We now finally turn to consider groups with finite center.  Let $G$ have a finite center $Z$, and let $(X,\mu)$ be an ergodic p.m.p. action. The function space $L^2(X)$  decomposes into a finite direct sum of closed subspaces $\cV_\chi$, with $\chi$ ranging over the characters of the finite Abelian group $Z$, and for a function $f\in \cV_\chi$, we have $f(zx)=\chi(z)f(x)$. $Z$ being central, each space $\cV_\chi$ is $G$-invariant, and we are left with showing the pointwise convergence of the averages in question 
for functions in each $\cV_\chi$, $\chi\in Z^\ast$. But since $U$ and $V$ are $Z$-invariant sets and $Z$ is central, it follows from the fact that $\sum_{z\in Z}\chi(z)=0$ for a character $\chi\neq 1$ of $Z$, that the action of the corresponding averages annihilate each $\cV_\chi$, unless $\chi$ is the trivial character, denoted $1$.  The functions in $\cV_\chi$, $\chi\neq 1$ have zero integral on $X$, so the pointwise ergodic theorem holds in these subspaces. 

The function space $\cV_1$ consists of $Z$-invariant functions on $X$, and is naturally identified with the function space $L^2(X/Z)$, where $X/Z$ is the space of orbits of $Z$ in $X$. The representation of $G$ on this function space is such that center acts trivially, and it is equivalent to the representation that arises from the action of $G/Z$ on the space $X/Z$ of $Z$-orbits in $X$. The desired convergence results then follow from our previous arguments for the adjoint group $G/Z$. 

This concludes the proof of Theorem  \ref{thm:main}. \qed

\begin{rem}\label{general sets} {\it Case of general averages}.
We note that the assumption that the sets $U,V\subset K$ are $Z$-invariant is not strictly necessary. Let us briefly outline a proof of the pointwise ergodic theorem for general sets $U,V \subset K \subset G$. Unlike the rest of our discussion throughout this paper, the argument we indicate here is not geometric, but rather spectral in nature. 

 Clearly, for any $f \in \cV_\chi$ defined above, the absolute value $\abs{f(x)}$ is a $Z$-invariant function. It follows immediately that the averages defined by $U$ and $V$ satisfy the strong maximal inequality in $L^p$, $p > 1$ and $L\log L$. Indeed, clearly 
 $$\abs{\pi(\sigma_{r,\epsilon}^{U,V})f}\le \pi(\sigma_{r,\epsilon}^{U,V})\abs{f}=\pi(\bar{\sigma}_{r,\epsilon}^{U,V})\abs{f}$$
where $\bar{\sigma}_{r,\epsilon}^{U,V}$ is the projection of the measure $\sigma_{r,\epsilon}^{U,V}$ on $G$  to $G/Z$. Thus the maximal inequalities for the operators $\bar{\sigma}_{r,\epsilon}^{U,V}$ on $G/Z$, acting on $X/Z$,  imply the desired result. 
  
We are then left with showing that there is a dense subspace in $\cV_\chi$ where pointwise almost sure convergence of our averages occurs. This fact can be deduced using the argument appearing in section 2.5 of \cite{N}. There,  pointwise convergence for a suitable spectrally defined dense set of $K$-finite functions is established for $SL_2(\R)$, based on derivative estimates for $K$-finite functions. Similar derivative estimates can be established for all real-rank-one groups with finite center, using e.g. the results in \cite{Co} or \cite{CM}. This establishes pointwise almost sure convergence in a dense subspace of $L^2(X)$, and by a routine application of the maximal inequalities the pointwise ergodic theorem holds as stated. 
\end{rem}










\begin{thebibliography}{1000000}



\bibitem[Bi32]{Bi} Birkhoff,  G. D., \textit{Proof of the ergodic theorem}, Proc. Nat. Acad. Sci. USA \textbf{17} (1931), 656--660.


\bibitem[BK96]{BK} Becker, H. and Kechris, A.S. \textit{The descriptive set theory of Polish group actions}, volume 232 of London Mathematical Society Lecture Note Series. Cambridge University Press, Cambridge, 1996.

\bibitem[BN13]{BN13} Bowen, L. and Nevo, A., \textit{Geometric covering arguments and ergodic theorems for free groups}.  L'Enseignement Math\'ematique, \textbf{59}, 2013, pp. 133--164.


\bibitem[BN14]{BN14} Bowen, L. and Nevo, A., \textit{Amenable equivalence relations and the construction of ergodic averages for group actions}, to appear in Journal d'Analyse Math\'ematique.

\bibitem[BH99]{BH} Bridson, M. and  Haefliger, A. \underline{Metric spaces of non-positive curvature}.  Series of Comprehensive Studies in Mathematics. vol. 319, Springer-Verlag, 1999.  


\bibitem[Co]{Co} Cowling, M. \textit{Sur les coefficients des representations unitaires des groupes de Lie simples.
Analyse harmonique sur les groupes de Lie}. S\'eminaire Nancy Strasbourg 1975, pp. 132-178, Lecture Notes in Mathematics, 739, Springer Verlag, 1979.

\bibitem[CM]{CM} Casselman, W.  and Milicic, D., \textit{Asymptotic behavior of matrix coefficients of admissible representations}. Duke Math. J. 49 (1982), no. 4, 869?930

\bibitem[Fa82]{Fa72} Faraut, J., \textit{Analyse harmonique sur les espaces hyperboliques.} (French) [Harmonic analysis on hyperbolic spaces] Topics in modern harmonic analysis, Vol. I, II (Turin/Milan, 1982), 445--473, Ist. Naz. Alta Mat. Francesco Severi, Rome, 1983. 

\bibitem[Fav72]{Fav72} Fava, N. A., 
\textit{Weak type inequalities for product operators.} 
Studia Math. 42 (1972), 271--288.

\bibitem[GN10]{GN10} Gorodnik, A. and Nevo, A. \textit{ The ergodic theory of lattice subgroups}. Annals of Mathematics
Studies, 172. Princeton University Press, Princeton, NJ, 2010.



\bibitem[He84]{He1} Helgason, S.  \textit{Groups and Geometric Analysis}, Acadmic Press, 1984. 


\bibitem[He94]{He2} Helgason, S. \textit{Geometric Analysis on Symmetric Spaces}, Mathematical surveys and monographs {\bf 39}, American Math. Soc. 1994. 

\bibitem[HM79]{HM79} Howe, R. and  Moore C. C.,
Asymptotic properties of unitary representations. 
J. Funct. Anal. \textbf{32} (1979), no. 1, 72--96.


\bibitem[Io99]{Io99} Ionescu, A. \textit{An endpoint estimate for the Kunze-Stein phenomenon and related maximal operators}. Ann. of Math. \textbf{152} (2000), no. 1, 259--275. 

\bibitem[Ke10]{Ke10} Kechris, A. S. \textit{Global aspects of ergodic group actions}, volume 160 of Mathematical Surveys and Monographs. American Mathematical Society, Providence, RI, 2010.

\bibitem[Kn96]{Kn} Knapp, A. W., \textit{Lie groups beyond an introduction}. Progress in Mathematics \textbf{140}, Birkhausr, 1996. 


\bibitem[Ko84]{Koo84} Koornwinder T, H., \textit{Jacobi functions and analysis on noncompact semisimple Lie groups}. Special functions: group theoretical aspects and applications, 1Ð85, Math. Appl., Reidel, Dordrecht, 1984

\bibitem[Ko73]{Ko73} Kostant, B., \textit{On convexity, the Weyl group and the Iwasawa decomposition}. Annales ENS \textbf{6} (1973), 413-455. 

 \bibitem [Ne94]{N94} Nevo, A., 
\textit{Pointwise ergodic theorems for radial averages on simple
        Lie groups I}. Duke Math. J. \textbf{76} (1994), 113--140.

\bibitem [Ne97]{N97} Nevo, A., 
\textit{Pointwise ergodic theorems for radial averages on simple
        Lie groups II}. Duke  Math. J. \textbf{86} (1997), 239--259.

 \bibitem[Ne98]{Ne98} Nevo, A. 
\textit{Spectral transfer and pointwise ergodic theorems for semi-simple Kazhdan groups.}  
Math. Res. Lett. 5 (1998), no. 3, 305--325.

                                                                                                                                    



\bibitem[Ne05]{Ne05} Nevo, A., \textit{Pointwise ergodic theorems
for actions of groups}. Handbook of Dynamical Systems,
vol. 1B, Eds. B. Hasselblatt and A. Katok, 2006, Elsevier, pp. 871-982.



\bibitem[N]{N}	Nevo, A.  \textit{Equidistribution in measure-preserving actions of semisimple groups : case of $SL_2(\mathbb{R})$}. Math. arXiv:1708.03886, August 2017. 




\bibitem [NS97]{NS97} Nevo, A and Stein, E. M. 
 Analogs of Wiener's ergodic theorems for semi-simple Lie
        groups I. Ann. of Math., \textbf{145} (1997), pp. 565--595. 





\bibitem[Pi42]{Pi42} Pitt, H. R., \textit{Some generalizations of the ergodic theorem}. Mathematical Proceedings of the Cambridge Philosophical Society, \textbf{38}, 325-343, 1942. 

\bibitem[Wi39]{Wi39} Wiener, N., \textit{ The ergodic theorem}. Duke Math. J. \textbf{5} (1939), 1--18.



\end{thebibliography}
\end{document}